\documentclass[leqno]{article}
\pdfoutput=1

\usepackage{graphicx}
\usepackage{amsmath,amstext,amssymb,amsthm}
\numberwithin{equation}{section}
\newtheorem{theorem}{Theorem}[section]
\newtheorem{lemma}[theorem]{Lemma}

\newtheorem{remark}{Remark}[section]
\newcommand{\abs}[1]{\left\vert#1\right\vert}

\newcommand{\norm}[1]{\left\Vert#1\right\Vert}
\newcommand{\norml}[2]{\left\Vert#1\right\Vert_{L^2(#2)}}

\newcommand{\norme}[1]{\left\Vert#1\right\Vert_{E}}
\newcommand{\normc}[1]{\left\Vert\hskip -0.8pt \left\vert #1 \right\vert\hskip -0.8pt\right\Vert_{1,h,H}}

\newcommand{\pd}[1]{\left\langle #1\right\rangle}

\newcommand{\set}[1]{\left\{#1\right\}}
\newcommand{\av}[1]{\left\{#1\right\}}
\newcommand{\jm}[1]{\left[#1\right]}

\newcommand{\M}{\mathcal{M}}

\newcommand{\R}{\mathbb{R}}

\newcommand{\bn}{\mathbf{n}}

\newcommand{\nn}{\nonumber}
\newcommand{\ls}{\lesssim}

\newcommand{\al}{\alpha}
\newcommand{\be}{\beta}

\newcommand{\ep}{\epsilon}
\newcommand{\eps}{\epsilon}
\newcommand{\ga}{\gamma}
\newcommand{\Ga}{\Gamma}

\newcommand{\na}{\nabla}

\newcommand{\Om}{\Omega}
\newcommand{\pa}{\partial}

\newcommand{\rd}{\,\mathrm{d}}

\DeclareMathOperator{\dist}{{dist}}
\DeclareMathOperator{\supp}{{supp}}

\title{A combined finite element and multiscale finite element method for the multiscale elliptic
problems}
\author{
Weibing Deng
\thanks{Department of Mathematics, Nanjing University, Jiangsu,
210093, P.R. China. ({\tt wbdeng@nju.edu.cn}). The work of this author was
partially supported by the the NSF of China grant 10971096 and by the Fundamental Research Funds for the Central Universities 1116020306.}
\and
Haijun Wu\thanks{Department of Mathematics, Nanjing University, Jiangsu,
210093, P.R. China. ({\tt hjw@nju.edu.cn}). The work of the second author was
partially supported by the National Magnetic Confinement Fusion Science Program under grant 2011GB105003 and by the NSF of China grants 11071116, 91130004.}
}

\begin{document}
\date{}
\maketitle


\setcounter{page}{1}

\large
\begin{abstract}
The oversampling multiscale finite element method (MsFEM) is one of the most popular
methods for simulating composite materials and flows in porous media which may have many scales.
But the method may be inapplicable or inefficient in some portions of the computational domain, e.g., near the domain boundary or near long narrow channels inside the domain due to the lack of permeability information outside of the domain or the fact that the high-conductivity features cannot be localized within a coarse-grid block. In this paper we develop a combined finite element and multiscale finite element method (FE-MsFEM), which deals with such portions by using the standard finite element method on a fine mesh and the other portions by the oversampling MsFEM. The transmission conditions across the FE-MSFE interface is treated by the penalty technique.
A rigorous convergence analysis for this special FE-MsFEM is given under the assumption that the diffusion coefficient is periodic. Numerical experiments are carried out for the elliptic equations with periodic and random highly oscillating coefficients, as well as multiscale problems with high contrast channels, to demonstrate the accuracy and efficiency of the proposed method.
\end{abstract}



\section{Introduction}
 Let $\Omega\subset\R^n, n=2,3 $ be a polyhedral domain, and consider the following elliptic equation
\begin{equation}\label{eproblem}
\left\{\begin{aligned} -\nabla\cdot(\mathbf{a}^\ep(x)\nabla
u_\ep(x)) & = f(x)&&
  \text{in}\,\Omega, \\
  u_\ep(x)&=0&&   \text{on}\,\partial\Omega,
\end{aligned}\right.
\end{equation}
where $0<\epsilon\ll 1$ is a parameter that represents the ratio of the
smallest and largest scales in the problem, and $\mathbf{a}^\ep(x)=(a_{ij}^\ep(x))$ is a
symmetric, positive definite, bounded tensor:
\begin{equation}\label{eq13}
  \lambda |\xi|^2\leq a_{ij}^\ep(x)\xi_i\xi_j\leq\Lambda |\xi|^2\quad
  \forall \xi\in \R^n,\, x\in\bar{\Omega}
\end{equation}
for some positive constants $\lambda$ and $\Lambda$.

Problems of the type
(\ref{eproblem}) are often used to describe the models arising from composite materials and
flows in porous media, which contain many spatial scales. Solving these problems numerically is difficult because of that resolving the smallest scale in problems usually
requires very fine meshes and hence tremendous amount of computer memory
and CPU time.  To overcome this difficulty, many methods have been
designed to solve the problem on meshes that are coarser than the
scale of oscillations. One of the most popular methods is the multiscale
finite element method (MsFEM) \cite{EHW,HW,HWC}, which  takes its origin from the work of Babu\v{s}ka and Osborn \cite{BO1983,BCO1994}. Two main ingredients of the MsFEM are the global formulation of the method such as various finite element methods and the construction of basis functions. The special basis functions which constructed from the local solutions
of the elliptic operator contain the small scale information within each element. By solving the problem (\ref{eproblem})
in the special basis function space, they get a good approximation of the
full fine scale solution. We remark that there are many other methods proposed to solve this type of multiscale problems in the past several decades. See, for instance,
wavelet homogenization techniques \cite{DE,ER}, multigrid
numerical homogenization techniques \cite{FB2,MDH}, the subgrid
upscaling method \cite{Arbo1,Arbo2}, the heterogeneous
multiscale method \cite{EE1,EE3,EMZ}, the residual-free bubble method (or the
variational multiscale method, discontinuous enrichment
method) \cite{BFHR,FHF,FY,FR,Hu,San}, mortar multiscale methods \cite{APWY2007,PWY2002}, and upscaling or numerical homogenization method \cite{dur,Farmer,WEH}. We refer the reader to the book
\cite{EH2009} for an overview and more other references of multiscale numerical methods in the literature, especially a description of some intrinsic connections between most of these methods.

In this paper, we focus on the MsFEM. Many developments and extensions of the MsFEM have been done in the past ten years. See for example, the mixed MsFEM \cite{CH2002, Aarnes2004}, the MsFEMs for nonlinear problems \cite{EHG2004, EP2004}, the Petro-Galerkin MsFEM \cite{HWZh}, the MsFEMs using limited global information \cite{EGHE2006, OZh2007},  and the multiscale finite volume method \cite{JLT2003}.  In \cite{HWC}, it is shown that there is a resonance error between the grid scale and the scales
of the continuous problem. Especially, for the two-scale problem, the resonance error manifests as a
ratio between the wavelength of the small scale oscillation and the grid size; the error
becomes large when the two scales are close. The scale resonance is a fundamental
difficulty caused by the mismatch between the local construction of the multiscale basis
functions and the global nature of the elliptic problems. This mismatch between the
local solution and the global solution produces a thin boundary layer in the first order
corrector of the local solution.  To overcome the difficulty due to the scale resonance, an oversampling
technique was proposed in \cite{HW, EHW}. The basic idea is computing the local problem in the domain with size larger than the mesh size $H$ and use only the interior
sampled information to construct the basis functions.  By doing this, the influence of the boundary layer in the larger domain on the basis
functions is greatly reduced.

However, for the coarse-gird elements near the boundary, in order to construct the multiscale basis functions, the oversamping MsFEM needs to assume that there is enough information
available outside of the research domain, which is not applicable in practice.
To handle this problem, the natural way is to use the standard  multiscale basis functions instead of the oversampling multiscale basis functions in the coarse-grid elements adjacent to the boundary, hence in this area we don't need to use the information outside the domain.
We call this method as the mixed basis MsFEM. Since in the elements near the boundary,
we use the multiscale basis functions without oversampling technique,
the scale resonance comes out again hence pollute the accuracy.

To overcome this difficulty, we introduce a new method in this paper which can improve the accuracy significantly. The proposed method separates the research area into two sub-domains such that one of them is contained inside the domain with a distance away from the boundary. Then in the interior sub-domain the oversampling multiscale basis functions on a coarse mesh (with mesh size $H$) are used. While, in the other sub-domain which adjacent to the boundary the traditional linear FEM basis functions are used on a mesh (with mesh size $h$) which is fine enough to resolving multiscale features. The difficulty to realize this idea is how to joint the two methods together without losing accuracies of both methods, i.e., how to deal with the transmission condition on the interface between coarse and fine meshes efficiently. Thanks to the penalty techniques used in the interior penalty discontinuous (or continuous) Galerkin methods originated in 1970s \cite{bz73,baker77,dd76,w78,arnold82,abcm01}, we may deal with the transmission condition on the interface by penalizing the jumps from the function values as well as the fluxes of the finite element solution on the fine mesh to those of the oversampling multiscale finite element solution on the coarse mesh. A rigorous and careful analysis is given for the elliptic equation with periodic diffusion coefficient to show that the $H^1$-error of our new method is just the sum of interpolation errors of both methods plus an error term of $O(\frac{H^2}{\sqrt{\ep}})$ introduced by the penalty terms, where $H$ is the mesh size of the coarse mesh. We would like to remark that besides the applications of penalty technique to the interior penalty discontinuous (or continuous) Galerkin methods, this technique is also applied to the Helmholtz equation with high wave number to reduce the pollution error \cite{fw09, fw11, w, zw} and applied to the interface problems to construct high order unfitted mesh methods\cite{m09, wx10}.

The other potential application of our proposed method is to solve the multiscale problems which may have some singularities. For example, the multiscale problem with Dirac function singularities, which stems from the simulation of steady flow transport through highly heterogeneous porous media driven by extraction wells \cite{CY2002}, or the multiscale problems with high-conductivity channels that connect the boundaries of coarse-grid blocks \cite{GE20101,GE20102,EGW2011,OZh2011}. Our new FE-MsFEM may solve such problems by using the traditional FEM on a fine mesh near the singularities (and, of course, near the domain boundary) and using the oversampling MsFEM in the other part of the domain. To demonstrate the performance of the FE-MsFEM, we try to simulate  multiscale elliptic problems which have  fine and long-ranged high-conductivity channels. We remark that this kind of high-conductivity features cannot be localized within a coarse-grid block, hence it is difficult to be handled with standard or oversampling multiscale basis. The numerical results show that the introduced FE-MsFEM can solve the high contrast multiscale elliptic problems efficiently.
The convergence analysis for multiscale problems with singularities and applications of the proposed FE-MsFEM to practical problems such as two-phase flows in porous media and other types of equations are currently under study.

The rest of this paper is organized as follows. In Section~\ref{sec-2}, we formulate the FE-MsFEM for the model problem. In Section~\ref{sec-3}, we review some classical homogenization results for the
elliptic problems and give an interior $H^2$ norm error estimate between the multiscale solution and the homogenized solution with first order corrector.  In Section~\ref{sec-4}, we give some approximation properties for the oversampling MsFE space and the linear FE space, respectively. The $H^{1}$ error estimate of the introduced FE-MsFEM is given in Section~\ref{sec-5}.
In Section~\ref{sec-6} we first give
some numerical examples for both periodic and randomly generated coefficients to demonstrate the accuracy the proposed method, and then apply our method to multiscale elliptic problems which have fine and long-ranged high-conductivity channels to demonstrate the efficiency of the method. Conclusions are drawn in the last section.

Before leaving this section, we fix some notations and conventions
to be used in this paper. In the following, the Einstein summation
convention is used: summation is taken over repeated indices.
$L^2(\Omega)$ denotes the space of square integrable functions
defined in domain $\Omega$. We use the $L^2(\Omega)$ based Sobolev
spaces $H^k(\Omega)$ equipped with norms and seminorms given by:
$$
\|u\|_{H^{k}(\Omega)}^2=\int_\Omega\sum_{|\alpha|\leq k}|D^\alpha
u|^2,\quad  |u|_{H^k(\Omega)}^2=\int_\Omega\sum_{|\alpha|= k}|D^\alpha
u|^2.
$$
$\|u\|_{W^{k,\infty}(\Omega)}$ ($|u|_{W^{k,\infty}(\Omega)}$) is the
$W^{k,\infty}$ norm (seminorm) of $u$ in $\Omega$.  Throughout, $C,
C_1, C_2,\cdots$ denote generic constants, which are independent of
$\ep$, $H$ and $h$ unless otherwise stated. We also use the shorthand notation
$A\lesssim B$ and $B\gtrsim A$ for the inequality $A\leq C B$ and $B\geq CA$. The notation $A\eqsim B$ is equivalent to the statement
$A\ls B$ and $B\ls A.$

\section{FE-MsFEM Formulation}\label{sec-2}

In this section we present our FE-MsFEM. We describe the method only for the case of dealing with the difficulty of lack information outside the domain in the oversampling MsFEM. Of course, the formulation can easily be extended to the case of dealing with singularities.

We first separate the research area $\Omega$ into two sub-domains $\Omega_1$ and $\Omega_2$ such that $\Om_2\subset\subset\Om$ and $\Omega=\Omega_1\cup\Omega_2\cup\Gamma$, where $\Gamma=\partial\Omega_1\cap\partial\Omega_2$ is the interface of $\Omega_1$ and $\Omega_2$ (cf. Fig.~\ref{fig:1}). For simplicity, we assume that the length/area of $\Ga$ satisfies $|\Ga|= O(1)$.
Let ${\cal M}_h$ be a triangulation of the domain $\Omega_1$ and ${\cal M}_H$ be a  triangulation of the domain $\Omega_2$, and denote $\Ga_h$ and $\Ga_H$ the two partitions of the interface $\Ga$ induced by ${\cal M}_h$ and ${\cal M}_H$, respectively. We assume that on the interface $\Ga$, ${\cal M}_H$ and ${\cal M}_h$ satisfy the matching condition that $\Ga_h$ is a refinement of $\Ga_H$.  Clearly, each edge/face  in $\Gamma_{{H}}$ is composed of some edges/faces in $\Gamma_{{h}}$. Combining the two triangulations together, we define ${\cal M}_{h,H}$ as the triangulation of $\Omega$ (See Fig.~\ref{fig:1} for an illustration of triangulation). For any element $K\in{\cal M}_{h}$ (or $K\in{\cal M}_{H}$), we define $h_K$ (or $H_K$) as $\mathrm{diam}(K)$. Similarly, for each edge/face $e$ of $K_e\in{\cal M}_{h}$ (or $E$ of $K_E\in{\cal M}_{H}$ ), define $h_e$ as $\mathrm{diam}(e)$ (or $H_E$ as $\mathrm{diam}(E)$). Denote by $h=\max_{K\in\M_h}h_K$ and $H=\max_{K\in\M_H}H_K$. We assume that $h<\epsilon<H$, that $\{\M_H\}$ and $\{\M_h\}$ are shape-regular and quasi-uniform.

\begin{figure}[htp]
\centerline{\includegraphics[width=0.5\textwidth]{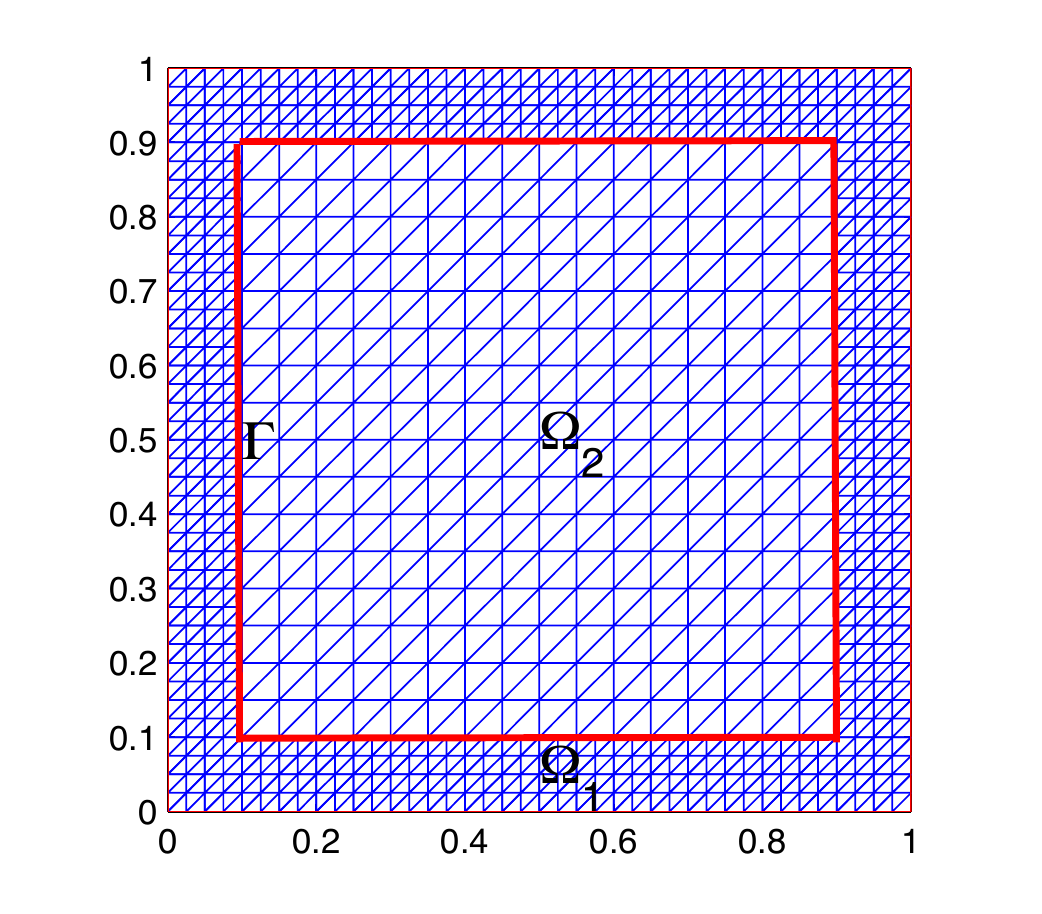}}
\caption{A separation of the domain and a sample mesh. }
\label{fig:1}
\end{figure}

 For any point on $\Ga$, we associate a unit normal $\mathbf{n}$, which is oriented from $\Omega_1$ to $\Omega_2$.
We also define the jump $\jm{v}$  and average $\av{v}$ of $v$ on the interface $\Ga$
 as
 \begin{equation}\label{eja}
    \jm{v}:=v|_{\Om_1}-v|_{\Om_2},\qquad \av{v}:=\frac{v|_{\Om_1}+v|_{\Om_2}}{2}.
 \end{equation}

  Introduce the  ``energy" space
\begin{equation}\label{eV}
    V:= \set{v: \;v|_{\Om_i}=v_i,\; \text{ where } v_{i}\in H_0^1(\Om)\cap H^s(\Om),  \,i=1, 2}, \text{ for some } s>\frac32.
\end{equation}
Testing the elliptic problem \eqref{eproblem} by any $v\in V$, using integration by parts, and using the identity $\jm{vw}=\av{v}\jm{w}+\jm{v}\av{w}$,  we obtain
\begin{align*}
    \sum_{i=1}^2 \int_{\Om_i}\mathbf{a}^\ep\na u_\ep\cdot\na v -\int_\Ga \av{\mathbf{a}^\ep\na u_\ep\cdot\bn}\jm{v} =\int_{\Om} fv .
\end{align*}
Define the bilinear form $A_\be(\cdot,\cdot)$ on $V\times V$:
\begin{align}
A_\beta(u,v) :=&\sum_{K\in {\cal M}_{h,H}} \int_{K}\mathbf{a}^\ep\na u\cdot\na v \label{eah}\\
&-\sum_{e\in\Ga_h}\int_{e} \Big(\av{\mathbf{a}^\ep\na u\cdot\bn}\jm{v}
+\be \jm{u}\av{\mathbf{a}^\ep\na v\cdot\bn}\Big) \\
&+  J_0(u,v)+  J_1(u,v),\nn \\
J_0(u,v):=&\sum_{e\in\Ga_h}\frac{\ga_0}{\rho}\int_{e} \jm{u}\jm{v} ,\label{eJ0}\\
J_1(u,v):=&\sum_{e\in\Ga_h}{\ga_1\,\rho}\int_e
\jm{\mathbf{a}^\ep\na u\cdot\bn}\jm{\mathbf{a}^\ep\na v\cdot\bn} ,\label{eJ1}
\end{align}
where $\be$ is a real number such as $-1,0,1$, and $\ga_0, \ga_1, \rho>0$ will be specified later.  Define further the linear form $F(\cdot)$ on $V$:
\begin{equation*}
 F(v):=\int_{\Om} fv .
\end{equation*}
It is easy to check that the solution $u_\ep$ to the problem \eqref{eproblem} satisfies the following formulation:
\begin{equation}\label{evp}
A_\beta(u_\ep,v) =F(v) \qquad\forall\, v\in V.
\end{equation}

To formulate the FE-MsFEM, we need the oversampling MsFE space on ${\cal M}_H$ defined as follows (cf. \cite{CW2010,HW,EH2009}).
For any $K\in {\cal M}_H$ with nodes $\{x_i^K\}_{i=1}^{n+1}$, let
$\{\varphi_i^K\}^{n+1}_{i=1}$ be the basis of $P_1(K)$ satisfying
$\varphi_i^K(x_j^K)=\delta_{ij}, 1\le i,j\le n+1$, where $\delta_{ij}$ stands for the Kronecker's symbol. For any $K\in {\cal M}_H,$ we denote by
$S=S(K)$ a macro-element (simplex) which contains $K$ and satisfies that
$H_S\leq C_1H_K$ and $\text{dist} (K,\partial S)\geq C_0 H_K$, where
 $C_1>0$ is independent of $H_K$ and $C_0$ is from \eqref{eC0}. We assume that the macro-elements $S(K)$ are also shape-regular. Denote by $\{\varphi_i^S\}_{i=1}^{n+1}$ the nodal basis of $P_1(S)$ such
that $\varphi_i^S(x_j^S)=\delta_{ij}, 1\le i,j\le n+1,$ where $x_j^S$ are vertices of $S$.

Let $\psi_i^S\in H^1(S), i=1,\cdots,n+1,$  be the solution of
the problem
\begin{equation}\label{overbase}
  -\nabla\cdot (\mathbf{a}^\ep\nabla\psi_i^S)=0\qquad \text{ in } S,\qquad \psi_i^S|_{\partial S}=\varphi_i^S.
\end{equation}
 The oversampling
multiscale finite element basis functions over $K$ is defined by
\begin{equation}\label{coefcij1}
 \bar{ \psi_i}^K=c_{ij}^K\psi_j^S|_K\qquad \text{ in } K,
\end{equation}
with the constants so chosen that
\begin{equation}\label{coefcij2}
  \varphi_i^K=c_{ij}^K\varphi_j^S|_K\qquad \text{ in } K.
\end{equation}
The existence of the constants $c_{ij}^K$ is guaranteed because
$\{\varphi_j^S\}_{j=1}^{n+1}$ also forms a basis of $P_1(K)$.

Let $\text{OMS}(K)=\mathrm{span\,}\{\bar{\psi_i}^K\}_{i=1}^{n+1}$ be the set of space functions on $K$. Define
the projection $\Pi_K:\text{OMS}(K)\rightarrow P_1(K)$ as
\begin{equation*}
  \Pi_K\psi=c_i\varphi_i^K\qquad \text{ if }\qquad \psi=c_i\bar{\psi_i}^K\in \text{OMS}(K).
\end{equation*}
Introduce the space of discontinuous piecewise ``OMS" functions  and the space of discontinuous piecewise linear functions:
\begin{align*}
  \bar{X}_H&=\{\psi_H:\; \psi_H|_K\in\text{OMS}(K)\ \ \forall K\in {\cal M}_H \},\\
  \bar{W}_H&=\{w_H:\; w_H|_K\in P_1(K)\ \ \forall K\in {\cal M}_H \}.
\end{align*}
Define $\Pi_H:\bar{X}_H\rightarrow\bar{W}_H$ through the relation
 \begin{equation}
 \Pi_H\psi_H|_K=\Pi_K\psi_H\ \ \mbox{for any $K\in{\cal M}_H , \psi_H\in\bar{X}_H$}.
 \end{equation}
The oversampling multiscale finite element space on $\M_H$ is then defined as
 \begin{equation*}
  X_H=\{\psi_H\in\bar{X}_H:\Pi_H\psi_H\in W_H\subset H^1(\Omega_2)\},
\end{equation*}
where $W_H=\bar{W}_H\cap H^1(\Om_2)$ is the $H^1$-conforming linear finite element space over $ {\cal M}_H$. In general, $X_H\not\subset H^1(\Omega_2)$ and the requirement
$\Pi_H\psi_H\in W_H$ is to impose certain continuity of the
functions $\psi_H\in X_H$ across the inter-element boundaries. According to the definition of $\Pi_K$, we have $\Pi_K \bar{\psi_i}^K= \varphi_i^K$. Since $\varphi_i^K$ is continuous across the element, this above requirement $\Pi_H\psi_H\in W_H$ is satisfied naturally since in each node we only have one freedom ( unknowns ).

Denote by $W_h$ the $H^1$-conforming linear finite element space over $ {\cal M}_h$ and by
\begin{equation}\label{wh0}
    W_h^0:=\set{w_h\in W_h : \;w_h=0\;  \mathrm{on}\; \pa\Om_1/\Ga}.
\end{equation}

We define the FE-MsFE approximation
space $V_{h,H}$ as
\begin{equation}\label{eVhp}
V_{h,H}:=\set{v_{h,H}: \;v_{h,H}|_{\Om_1}=v_h,\; v_{h,H}|_{\Om_2}=v_H,\; \text{ where } v_{h}\in W_h^0,\; v_H\in X_H}.
\end{equation}
Note that $V_{h,H}\not\subset V$.
We are now ready to define
the FE-MsFEM inspired by the formulation
\eqref{evp}:
Find $u_{h,H}\in V_{h,H}$ such that
\begin{equation}\label{eipfem}
A_\beta(u_{h,H}, v_{h,H}) =F(v_{h,H}) \qquad\forall\,  v_{h,H}\in V_{h,H}.
\end{equation}

\begin{remark}\label{remark1} (a) If $\be=1$, then the the bilinear form $A_\be$ is symmetric and, as a consequence, the stiffness matrix is symmetric as well. If $\be\neq 1$, e.g., $\be=-1$, then the method is nonsymmetric.

(b) The parameter $\rho>0$ satisfies that $\rho\leq \ep$. In fact, it is chosen as $\ep$ in our later error analysis, while in practical computation, it may be chosen as the mesh size $h$.
\end{remark}

For further error analysis, we introduce several concepts related to the interface $\Ga$ and some discrete norms.
 Define the set of elements accompanying with the interface partition $\Ga_h$ (or $\Ga_H$ ) as following:
  \begin{align}
    K_{\Gamma_h}&:= \{ K\in {\cal M}_h:\; K \text{ has an edge/face  in } \Ga_h\},\label{eKGh}\\
    K_{\Gamma_H}&:= \{ K\in {\cal M}_H:\; K \text{ has an edge/face  in } \Ga_H\}.\label{eKGH}
  \end{align}
Clearly, the number of elements in  $K_{\Gamma_h}$ is $O(\frac{1}{h^{n-1}})$ and the number of elements in  $K_{\Gamma_H}$ is $O(\frac{1}{H^{n-1}})$. Moreover, we assume that
\begin{align}\label{eC0}
\text{dist}\{\Ga,\pa\Om\}\ge C_0 H\ge h+2\ep>0,
\end{align}
 where $C_0$ is a constant. Thus, we can define a narrow subdomain $\Om_\Ga\subset\subset\Om$ surrounding $\Ga$ as
      \begin{align}\label{eog}\Om_\Ga:=\Ga&\cup\{x:\, x\in\Om_1,\text{ dist}(x,\Ga)<h+2\ep\}\\
      &\cup\{x:\, x\in\Om_2,\text{ dist}(x,\Ga)<H+2\ep\}.\nn
       \end{align}
Denote by
\begin{align}\label{eogHh}
\Om_{\Ga_H}=\cup\{K:\,K\in K_{\Ga_H}\},\qquad \Om_{\Ga_h}=\cup\{K:\,K\in K_{\Ga_h}\}.
\end{align}
It is clear that $\Om_{\Ga_H}, \Om_{\Ga_h}\subset\Om_\Ga$ and $\text{dist}(\Om_{\Ga_H},\pa\Om_\Ga)$, $\text{dist}(\Om_{\Ga_h},\pa\Om_\Ga)\ge 2\ep$, repectively. 


 Denote by
 \begin{equation*}
 \norm{v}_{1,H}=\left(\sum_{K\in\M_H}\|(\mathbf{a}^\ep)^{1/2}\na v\|_{L^2(K)}^2\right)^{1/2}\ \ \ \ \forall
 \, v \in\prod_{K\in\M_H}H^1(K).
 \end{equation*}
We introduce the following energy norm and the broken norm on the space $V_{h,H}$:
\begin{align}
\norme{v}:=&\left(\norm{v}_{1,H}^2+ \norml{(\mathbf{a}^\ep)^{1/2}\na v}{\Om_1}^2\right)^{1/2},\label{enorma} \\
\normc{v}:=&\bigg(\norme{v}^2+\sum_{e\in\Ga_h} \frac{\ga_0}{\rho}\norml{\jm{v}}{e}^2+\sum_{e\in\Ga_h}{\ga_1\,\rho}\norml{\jm{\mathbf{a}^\ep\na v\cdot\bn}}{e}^2\label{enormb}\\
&\hskip 31pt+\sum_{e\in\Ga_h}
\frac{\rho}{\ga_0}\norml{\av{\mathbf{a}^\ep\na v\cdot
\bn}}{e}^2\bigg)^{1/2}. \notag
\end{align}

\section{The homogenization results}\label{sec-3}
In this section, we assume that
$\mathbf{a}^\ep(x)$ has the form $\mathbf{a}(x/\ep)$. Moreover, we assume that $a_{ij}(y)\in C_p^1(\mathbb{R}^n)$, where $C_p^1(\mathbb{R}^n)$ stands
for the collection of all $C^1(\mathbb{R}^n)$ periodic functions
with respect to the unit cube $Y$. It is shown that under these
assumptions (cf. \cite{BLP,JKO}), $u_\ep$ converges in a
suitable topology to the solution of the homogenized equation
 \begin{equation}\label{ehomoeqn}
\left\{\begin{aligned} -\nabla\cdot(\mathbf{a}^*\nabla u_0(x)) & =
f(x)&&
  \text{in}\,\Omega, \\
  u_0(x)&=0&&   \text{on}\,\partial\Omega,
\end{aligned}\right.
\end{equation}
where
\begin{equation}\label{eq15}
a_{ij}^*=\frac{1}{|Y|}\int_Y
a_{ik}(y)\left(\delta_{kj}+\frac{\partial\chi^j}{\partial
y_k}(y)\right)\rd y.
\end{equation}
Here $\chi^j$ is the periodic solution of the cell problem
\begin{equation}\label{eq16}
-\nabla_y\cdot(\mathbf{a}(y)\nabla_y
\chi^j(y))=\nabla_y\cdot(\mathbf{a}(y)\mathbf{e}_j),\quad
j=1,\cdots,n
\end{equation}
with zero mean, i.e., $\int_Y\chi^jdy=0$, and $\mathbf{e}_j$ is the
unit vector in the $j$th direction.
The variational form of the problem (\ref{ehomoeqn}) is to
find $u_0(x)\in H_0^1(\Omega)$ such that
\begin{equation}\label{e18}
(\mathbf{a}^*\nabla u_0, \nabla v)=(f,v) \quad\forall\, v\in H_0^1(\Omega).
\end{equation}
It can be shown that $\mathbf{a}^*$ is positive definite.
Thus by Lax-Milgram lemma, \eqref{e18} has a unique solution. If $f(x)\in L^2(\Om)$, from the regularity theory of elliptic equations, we have
\begin{equation}\label{regularityh2}
    \norm{u_0}_{H^2(\Om)}\leq C\norm{f}_{L^2(\Om)}.
\end{equation}

Let $\theta_\ep$
denote the boundary corrector which is
the solution of
\begin{equation}\label{e16}
\begin{aligned}
       -\na\cdot(\mathbf{a}^\ep\na\theta_\ep) & = 0 &&  \mbox{in }\Omega,\\
       \theta_\ep&=-\chi^j(x/\ep)\frac{\partial
u_0(x)}{\partial x_j} &&  \mbox{on
       }\pa\Omega.
\end{aligned}
\end{equation}
From the Maximum Principle, we have
\begin{equation}\label{maximumtheta}
    \norm{\theta_\ep}_{L^\infty(\Om)}\ls \abs{u_0}_{W^{1,\infty}(\Om)}.
\end{equation}

In the following part, for convenience's sake, we will set
\begin{equation}\label{eq206}
u_1(x, x/\ep)=u_0(x)+\ep\chi^j(x/\ep)\frac{\partial
u_0(x)}{\partial x_j}.
\end{equation}

The following error estimates are known (cf. e.g. \cite{CH2002, moskow1997first, CY2002}).

\begin{theorem}\label{homoerror}
Assume that $u_0\in H^2(\Om)\cap W^{1,\infty}(\Om)$. Then there exists a constant $C$
independent of $\ep$, the domain $\Om$, and the function $f$ such that
\begin{equation}\label{homoestimate}
\|\na(u_\ep-u_1-\ep\theta_\ep)\|_{L^2(\Om)}\leq
C\ep\abs{u_0}_{H^2(\Om)}.
\end{equation}
Moreover, the boundary corrector $\theta_\ep$ satisfies the estimate
\begin{equation}\label{boundarycorestimate}
\|\ep\na\theta_\ep\|_{L^2(\Om)}\leq
C\ep\abs{u_0}_{H^2(\Om)}+C\sqrt{\ep\abs{\pa\Om}}\abs{u_0}_{W^{1,\infty}(\Om)}.
\end{equation}
where $\abs{\pa\Om}$ stands for the measure of the boundary $\pa\Om$.
\end{theorem}

The following regularity estimate is an analogy of the classical interior estimate for elliptic equations in \cite[Theorem 8.8, P.183]{GT}.
\begin{lemma}\label{interiorestimate}
Let $w\in H^1(D)$ be the weak solution of the equation
\[
-\nabla\cdot(\mathbf{a}^\ep(x)\nabla w)  = g(x), \quad x\in D,
\]
where $\mathbf{a}^\ep(x)$ satisfies \eqref{eq13} and $g\in L^2(D)$. Then for any subdomain $D'\subset\subset D$, we have $w\in H^2(D')$ and
\begin{equation}\label{h2interiorestimate}
\abs{w}_{H^2(D')}\ls
\left(\frac{1}{d'}+\frac{1}{r}\right)\abs{w}_{H^1(D)}+\norm{g}_{L^2(D)}
\end{equation}
where $d'=\text{dist} (D',\partial D)$ and  $r$ is a constant such that $\abs{\mathbf{a}^\ep(x)-\mathbf{a}^\ep(y)}\le r\abs{x-y}$.
\end{lemma}
\begin{remark}
The $H^2$ norm interior estimate for elliptic equations is well-known in the literature. The importance in above estimate is the explicit dependence of the bound on $d'$ and $r$ which is crucial in our analysis.
\end{remark}
\begin{proof}
First, we define the difference quotient as follow:
\[
\Delta^hv(x)=\Delta_k^{h}v=\frac{v(x+he_k)-v(x)}{h},
\]
where $e_k, 1\le k\le n$ is the unit coordinate vector in the $x_k$ direction. And then, following the proof presented in \cite[Theorem 8.8, P.183]{GT}, we obtain
\[
\begin{split}
&\lambda\int_D\abs{\xi\na (\Delta^h w) }^2{\rm d} x\\
&\qquad\le \left(\frac{1}{r}\norm{\na w}_{L^2(D)}+\norm{g}_{L^2(D)}\right)\left(\norm{\xi\na(\Delta^h w)}_{L^2(D)}+2\norm{\Delta^h w\na\xi}_{L^2(D)}\right)\\
&\qquad\qquad+C \norm{\xi\na(\Delta^h w)}_{L^2(D)}\norm{\Delta^h w\na\xi}_{L^2(D)},
\end{split}
\]
where $\xi\in C_0^1(D)$  is the cut-off function such that
$0\leq\xi\leq 1$, $\xi=1$ in $D'$, and $|\na\xi|\leq 2/d'$ in $D$.
By use of the Young's inequality and the Lemma 7.23  in \cite[P.168]{GT}, it follows that
\[
\begin{split}
\norm{\xi\Delta^h(\na w)}_{L^2(D)}&\le C \left(\frac{1}{r}\norm{\na w}_{L^2(D)}+\norm{g}_{L^2(D)}+\norm{\Delta^h w\na\xi}_{L^2(D)}\right)\\
&\leq C \left(\frac{1}{r}+\frac{1}{d'}\right)\norm{\na w}_{L^2(D)}+C\norm{g}_{L^2(D)}.
\end{split}
\]
By the Lemma 7.24  in \cite[P.169]{GT} we obtain $\na w \in H^1(D')$, so that $w\in H^2(D')$ and the estimate \eqref{h2interiorestimate} holds. This completes the proof.
\end{proof}

Utilizing the above $H^2$ interior estimate to equation $-\nabla\cdot(\mathbf{a}^*\nabla u_{0{x_j}})  =f_{x_j}, j=1,\cdots,n$, where $x_j$ is the coordinate variable in the $j$th direction, we obtain

\begin{lemma}\label{h3u0} Let $u_0$ be the solution to \eqref{e18}.
Assume that $u_0\in H^2(D)$ and $f\in H^1(D)$. Then for any subdomain $D'\subset\subset D$ with $d'=\text{dist} (D',\partial D)$, we have
\begin{equation}\label{h3estimate}
\abs{u_0}_{H^3(D')}\ls
\frac{1}{d'}\abs{u_0}_{H^2(D)}+\norm{\na f}_{L^2(D)}.
\end{equation}
\end{lemma}

Further, by use of the $H^2$ interior estimate \eqref{h2interiorestimate}, we obtain an $H^2$ semi-norm interior estimate of the error $u_\ep-u_1$ in the narrow domain $\Om_\Ga$.
\begin{theorem}\label{h2homoerror}
Assume that $u_0\in H^2(\Om) \cap W^{1,\infty}(\Om)$, and $f|_{\Om_\Ga}\in H^1(\Om_\Ga)$. Then for any subdomain $\Om'\subset\subset \Om_\Ga$ with $\text{dist} (\Om',\partial \Om_\Ga)\ge 2\ep$, we have
\begin{equation}\label{h2homoestimate}
\abs{u_\ep-u_1}_{H^2(\Om')}\ls \abs{u_0}_{H^2(\Om)}+\ep\norm{\na f}_{L^2(\Om_\Ga)}+\frac{1}{\sqrt{\ep}}\abs{u_0}_{W^{1,\infty}(\Om)}.
\end{equation}
\end{theorem}

\begin{proof}
It is shown that, for any $\varphi\in H_0^1(\Om)$ (see \cite[p.550]{CH2002} or \cite[p.125]{CW2010}),
\begin{equation}\label{e20}
\begin{split}
 &\left(\mathbf{a}(x/\ep)\nabla(u_\ep-u_1),\nabla
 \varphi\right)_{\Om}\\
 &=(\mathbf{a}^*\nabla u_0, \nabla
 \varphi)_{\Om}-\left(\mathbf{a}(x/\ep)\nabla\left(u_0+\ep\chi^k\frac{\partial
 u_0}{\partial x_k}\right),\nabla\varphi \right)_{\Om}\\
 &=\ep\int_{\Om} a_{ij}(x/\ep)\chi^k
 \frac{\partial^2 u_0}{\partial x_j\partial
 x_k}\frac{\partial\varphi}{\partial x_i} -\ep\int_{\Om}
 \alpha_{ij}^k(x/\ep) \frac{\partial^2 u_0}{\partial
 x_j\partial x_k}\frac{\partial\varphi}{\partial x_i} ,
 \end{split}
\end{equation}
where
$\alpha_{ij}^k(y)\in H^1_{loc}(\R^n)$ are $Y$-periodic and dependent only on the coefficients $\mathbf{a}(y)$  (see \cite[p.6]{JKO}).
According to the assumption, there exists a subdomain $\Om_c\subset\subset\Om_\Ga$ such that $\Om'\subset\subset\Om_c$ with $d_0=\text{dist} (\Om',\partial \Om_c)\ge \ep$ and $d_1=\text{dist} (\Om_c,\partial \Om_\Ga)\ge \ep$.
From \eqref{e20}, it follows that in $\Om_c$,
\[
\na\cdot\left(\mathbf{a}(x/\ep)\nabla(u_\ep-u_1)\right)=\ep\frac{\pa}{\pa x_i}\left(a_{ij}(x/\ep)\chi^k
 \frac{\partial^2 u_0}{\partial x_j\partial
 x_k}-\alpha_{ij}^k(x/\ep) \frac{\partial^2 u_0}{\partial
 x_j\partial x_k}\right).
\]
Thus, from Lemma~\ref{interiorestimate}, it follows that
 \begin{equation*}
       \abs{u_\ep-u_1}_{H^2(\Om')}\ls  \Big(\frac{1}{d_0}+\frac{1}{\ep}\Big)\abs{u_\ep-u_1}_{H^1(\Om_c)}+ \abs{u_0}_{H^2(\Om_c)}+\ep\abs{u_0}_{H^3(\Om_c)}.
 \end{equation*}
Hence, from Theorem~\ref{homoerror} and Lemma~\ref{h3u0}, it follows the result \eqref{h2homoestimate} immediately.
\end{proof}


We conclude this section with a local $H^2$ semi-norm estimate for $u_1$ in $\Om_{\Ga_h}$, which will be used in the convergence analysis.

\begin{lemma}\label{uelocalh2}
Assume that $u_0\in H^2(\Om) \cap W^{1,\infty}(\Om)$, and $f|_{\Om_\Ga}\in H^1(\Om_\Ga)$. Then,
\begin{equation}\label{localh2esta}
    \abs{u_1}_{H^2{(\Om_{\Ga_h})}}\ls \abs{u_0}_{H^2(\Om_\Ga)}+ \ep^{-1} h^{1/2}\abs{u_0}_{W^{1,\infty}(\Om_{\Ga_h})}+\ep\|\na f\|_{L^2(\Om_\Ga)},
\end{equation}
where $\Om_{\Ga_h}$ and $\Om_\Ga$ are defined in \eqref{eogHh} and \eqref{eog}, respectively.
\end{lemma}
\begin{proof}
It is easy to see that
\[
\begin{split}
\abs{u_1}_{H^2{(\Om_{\Ga_h})}}&\le \abs{u_0}_{H^2{(\Om_{\Ga_h})}}+\abs{\ep\chi^j\frac{\pa u_0}{\pa x_j}}_{H^2{(\Om_{\Ga_h})}}\\
&\ls \abs{u_0}_{H^2{(\Om_{\Ga_h})}}+\ep^{-1}\abs{\frac{\pa u_0}{\pa x_j}}_{L^2{(\Om_{\Ga_h})}}+\ep\abs{u_0}_{H^3{(\Om_{\Ga_h})}}.
\end{split}
\]
Thus, from  Lemma~\ref{h3u0}, it follows that
\[
\abs{u_1}_{H^2{(\Om_{\Ga_h})}}\ls \abs{u_0}_{H^2(\Om_\Ga)}+\ep\|\na f\|_{L^2(\Om_\Ga)} +\ep^{-1}\abs{\Om_{\Ga_h}}^{1/2}\norm{\na u_0}_{L^\infty{(\Om_{\Ga_h})}},
\]
which, combining with the fact that $\abs{\Om_{\Ga_h}}=O(h)$, yields the result \eqref{localh2esta}.
\end{proof}
\section{Approximation properties of the FE-MsFE space $V_{h,H}$}\label{sec-4}
In this section, we give some approximation properties for the oversampling MsFE space $X_H$ and the linear FE space $W_h$, respectively.
\subsection{Approximation properties of oversampling MsFE space $X_H$}
We first recall a stability estimate for $\Pi_H$ (see \cite[Lemma 9.8]{CW2010} or \cite[Appendix B]{EHW}).
\begin{lemma}\label{L:8.8}
 There exist constants $\gamma$ and $C$ independent of $H$ and $\eps$ such that
 if $H_K\leq \gamma$ and $\eps/H_K\leq\gamma $ for all $K\in\M_H$, the following estimates are valid
 \begin{equation*}
   C^{-1}\norm{\na \Pi_H w_H}_{L^2(K)}\leq \norm{\na w_H}_{L^2(K)}
   \leq C\norm{\na \Pi_H w_H}_{L^2(K)}\ \ \ \ \forall\, w_H\in
   X_H.
 \end{equation*}
\end{lemma}

Next, we give some approximation properties of the oversampling MsFE space $\text{OMS}(K)$.
\begin{lemma}\label{lmultiapp} For any $K\in\M_H$, there exists $\phi_H\in \text{OMS}(K)$ such that the following estimates hold:
\begin{align}
\abs{u_1-\phi_H}_{H^1(K)}&\ls H_K\abs{u_0}_{H^2(K)}
   +{\eps}H_K^{n/2-1}\abs{u_0}_{W^{1,\infty}(K)},\label{lmultiappl1}\\
   \norm{u_1-\phi_H}_{L^2(K)}&\ls H_K^2\abs{u_0}_{H^2(K)}
   +{\eps}{H_K^{n/2}}\abs{u_0}_{W^{1,\infty}(K)},\label{lmultiappl2}\\
    \abs{u_1-\phi_H}_{H^2(K)}&\ls \ep^{-1}H_K\abs{u_0}_{H^2(K)}+H_K^{n/2-1}\abs{u_0}_{W^{1,\infty}(K)}
   +\ep\abs{u_0}_{H^3(K)}.\label{lmultiappl3}
\end{align}
 \end{lemma}
\begin{proof}
We take
   \begin{equation}\label{lintplMsFEM}
        \phi_H=\sum_{x_i^K\, \text{node of}\ K}u_0(x_i^K)\bar{\psi_i}(x),
   \end{equation}
then
   \begin{equation*}
      \Pi_K\phi_H=I_Hu_0,
   \end{equation*}
where $I_H:C(\bar{\Om}_2)\rightarrow W_H$ is the standard Lagrange
interpolation operator over linear finite element space. By the asymptotic expansion, we
know that
   \begin{equation}\label{expan1}
       \phi_H=I_Hu_0+\eps\chi^j\frac{\pa(I_Hu_0)}{\pa x_j}+\eps\theta_\eps^S,
   \end{equation}
where $\theta_\eps^S\in H^1(S)$ is the boundary corrector given by
   \begin{equation}\label{expan2}
    -\na\cdot(\mathbf{a}^\eps\na\theta_\eps^S)=0\quad \text{ in } S,\qquad \theta_\eps^S\big|_{\pa S}=-\chi^j\frac{\pa(I_Hu_0)}{\pa x_j}.
   \end{equation}
By the Maximum Principle we have
\begin{align}\label{thetainf}
 \norm{\theta_\eps^S}_{L^\infty(S)}\ls \abs{I_Hu_0}_{W^{1,\infty}(S)}\ls\abs{u_0}_{W^{1,\infty}(K)},
\end{align}
which together with the interior estimate in Avellaneda and Lin \cite[Lemma 16]{al87} imply that
     \begin{equation}\label{thetae}
       \norm{\na \theta_\eps^S}_{L^\infty(K)}\ls H^{-1}_K\norm{\theta_\eps^S}_{L^\infty(S)}\ls H_K^{-1}\abs{u_0}_{W^{1,\infty}(K)}.
    \end{equation}
Therefore
\begin{align}\label{E:k4}
   \norm{\na \Big(\phi_H-I_H u_0-\eps \chi^j\frac{\pa (I_H u_0)}{\pa x_j} \Big)}_{L^2(K)}
      &\ls \eps H_K^{-1}\abs{u_0}_{W^{1,\infty}(K)}\abs{K}^{1/2}\\
      &\ls \eps H_K^{n/2-1}\abs{u_0}_{W^{1,\infty}(K)}. \notag
\end{align}
Further, since
\begin{align*}
   &\norm{\na (u_0-I_Hu_0)}_{L^2(K)}\ls H_K\abs{u_0}_{H^2(K)},\\
   &\norm{\eps\na\Big( \chi^j\frac{\pa (u_0-I_H u_0)}{\pa x_j}\Big) }_{L^2(K)}\ls (H_K+\eps)\abs{u_0}_{H^2(K)},
\end{align*}
we obtain the result \eqref{lmultiappl1} immediately.

To prove the estimate \eqref{lmultiappl2}, we first notice that, from \eqref{thetainf},
   \begin{equation*}
       \norm{\theta_\eps^S}_{L^2(K)}\leq|K|^{1/2}\norm{\theta_\eps^S}_{L^\infty(K)}\ls H_K^{n/2}\abs{u_0}_{W^{1,\infty}(K)}.
    \end{equation*}
Therefore
\begin{equation}\label{e417}
   \norm{\phi_H-\left(I_H u_0+\eps \chi^j\frac{\pa (I_H u_0)}{\pa x_j}\right) }_{L^2(K)}
      \ls \eps H_K^{n/2} \abs{u_0}_{W^{1,\infty}(K)}.
\end{equation}
Further, since
\begin{align*}
   \norm{u_0-I_Hu_0}_{L^2(K)}&\ls H_K^2\abs{u_0}_{H^2(K)},\\
   \norm{\eps\chi^j\frac{\pa (u_0-I_H u_0)}{\pa x_j}}_{L^2(K)}&\ls \ep H_K\abs{u_0}_{H^2(K)},
\end{align*}
we obtain the result \eqref{lmultiappl2} immediately.

To prove the estimate \eqref{lmultiappl3}, it is easy to see that
\begin{align}
   &\abs{u_0-I_Hu_0}_{H^2(K)}\ls \abs{u_0}_{H^2(K)},\label{approxk1}\\
   &\abs{\eps \chi^j\frac{\pa (u_0-I_H u_0)}{\pa x_j}}_{H^2(K)}\ls (1+\ep^{-1}H_K)\abs{u_0}_{H^2(K)}+\ep\abs{u_0}_{H^3(K)}.\label{approxk2}
\end{align}
Let $D$ be a simplex such that
\[K\subset\subset D\subset\subset S \text{ and } \dist(K,\pa D),\; \dist(D,\pa S)\gtrsim H_K.\]
From Lemma~\ref{interiorestimate} and following the proof of \eqref{thetae}, we have
\[
\abs{\theta_\eps^S}_{H^2(K)}\ls (H_K^{-1}+\ep^{-1})\abs{\theta_\ep^S}_{H^1(D)}\ls H_K^{n/2-1}(H_K^{-1}+\ep^{-1})\norm{\na u_0}_{L^\infty(K)},
\]
which, combining with \eqref{approxk1} and \eqref{approxk2}, yields the result immediately. This completes the proof.
\end{proof}

From Lemma~\ref{lmultiapp}, we have the following local approximation estimates in $K_{\Ga_H}$.

 \begin{lemma}\label{multiapp} There exists $\psi_H\in X_H$ such that,
\begin{align}\label{multiapph2}
  &\bigg(\sum_{K\in K_{\Ga_H}} \norm{u_1-\psi_H}_{L^2(K)}^2 \bigg)^{1/2}\ls H^2\abs{u_0}_{H^2(\Om_{\Ga_H})}+\ep\sqrt{H}\abs{u_0}_{W^{1,\infty}(\Om_{\Ga_H})},\\
\label{multiapph3}
 & \bigg(\sum_{K\in K_{\Ga_H}} \abs{u_1-\psi_H}_{H^1(K)}^2 \bigg)^{1/2}\ls H\abs{u_0}_{H^2(\Om_{\Ga_H})}+\frac{\ep}{\sqrt{H}}\abs{u_0}_{W^{1,\infty}(\Om_{\Ga_H})}\\
 \label{multiapph4}
 & \bigg(\sum_{K\in K_{\Ga_H}} \abs{u_1-\psi_H}_{H^2(K)}^2 \bigg)^{1/2}\\
 &\qquad\qquad\quad\ls \frac{H}{\ep}\abs{u_0}_{H^2(\Om_\Ga)}+\ep\norm{\na f}_{L^2(\Om_\Ga)}+\frac{1}{
 \sqrt{H}}\abs{u_0}_{W^{1,\infty}(\Om_{\Ga_H})},\nn
\end{align}
where $\Om_{\Ga_H}$ is defined in \eqref{eogHh}.
\end{lemma}
\begin{proof}
We take
   \begin{equation}\label{intplMsFEM}
        \psi_H=\sum_{p_j \text{ node of}\ {\cal M}_H}u_0(p_j)\bar{\psi_j}(x)\in X_H,
   \end{equation}
then $\psi_H\Big|_K=\phi_H$, which is defined in Lemma~\ref{lmultiapp}. From \eqref{lmultiappl2}, it follows that
\[
\begin{split}
\sum_{K\in K_{\Ga_H}} \norm{u_1-\psi_H}_{L^2(K)}^2&\ls H^4\abs{u_0}_{H^2(\Om_{\Ga_H})}^2+\ep^2H^n\abs{u_0}_{W^{1,\infty}(\Om_{\Ga_H})}^2\sum_{K\in K_{\Ga_H}} 1 \\
&\ls H^4\abs{u_0}_{H^2(\Om_{\Ga_H})}^2+\ep^2H\abs{u_0}_{W^{1,\infty}(\Om_{\Ga_H})}^2,
\end{split}
\]
which yields \eqref{multiapph2} immediately. Note here we have used the fact that the number of elements in $K_{\Ga_H}$ is $O(\frac{|\Ga|}{H^{n-1}})$.
Similarly, \eqref{multiapph3} follows from~\eqref{lmultiappl1}.

It remains to prove \eqref{multiapph4}.
From \eqref{lmultiappl3} and Lemma~\ref{h3u0}, it follows that
\[
\begin{split}
\sum_{K\in K_{\Ga_H}} \abs{u_1-\psi_H}_{H^2(K)}^2&\ls \ep^{-2}H^2\abs{u_0}_{H^2(\Om_{\Ga_H})}^2+\ep^2\abs{u_0}_{H^3(\Om_{\Ga_H})}^2\\
&\quad+H^{n-2}\abs{u_0}_{W^{1,\infty}(\Om_{\Ga_H})}^2\sum_{K\in K_{\Ga_H}} 1\\
&\ls \left(\ep^{-2}H^2+1\right)\abs{u_0}_{H^2(\Om_\Ga)}^2+\ep^2\norm{\na f}_{L^2(\Om_\Ga)}^2+H^{-1}\abs{u_0}_{W^{1,\infty}(\Om_{\Ga_H})}^2,
\end{split}
\]
which yields \eqref{multiapph4} immediately.
The proof is completed.
\end{proof}

From Lemma~\ref{lmultiapp}, Theorem~\ref{homoerror}, via taking the same $\psi_H$ in Lemma~\ref{multiapp}, we also have the following result which gives an approximation estimate of the space $X_H$ (cf. \cite{EHW, EH2009, CW2010}).
 \begin{lemma}\label{multiapp1} There exists $\psi_H\in X_H$ such that,
\begin{equation}\label{multiapph1}
   \norm{u_\eps-\psi_H}_{1,H}\ls H\abs{u_0}_{H^2(\Om)}
   +\frac{\eps}{H}\abs{u_0}_{W^{1,\infty}(\Om_2)}+\sqrt{\ep}\abs{u_0}_{W^{1,\infty}(\Om)}.
   \end{equation}
\end{lemma}

\subsection{Approximation properties of linear FE space $W_h$ }
Since $|\Om_1|$, the area/volume of $\Om_1$, may be small, we prefer estimates with explicit dependence on it.
 To attain this aim, we use the Scott-Zhang interpolation instead of the standard FE interpolation in this subsection.

We first introduce the Scott-Zhang interpolation operator $Z_h:
H_E^1(\Omega_1) \rightarrow W_h^0$, where $H_E^1(\Om_1):=\{v\in H^1(\Om_1):\; v=0 \text{ on } \pa\Om_1\setminus\Ga\}$. For any node $z$ in $\M_h$, let $\phi_z(x)$ be the nodal basis function associated with $z$ and let $e_z$ be
an edge/face with one vertex at $z$, then the Scott-Zhang interpolation
operator is defined as \cite{scott1990finite}:
\begin{equation}\label{ec4}
Z_hv=\sum_{\text{node}\, z\, \text{in} \M_h}\bigg(\int_{e_z}\psi_zv\bigg)\phi_z
\quad \forall\, v\in H_E^1(\Omega_1),
\end{equation}
where  $\psi_z(x)$ is a linear function that satisfies
$\int_{e_z}\psi_z(x)w(x)=w(z)$ for any linear function $w(x)$ on
$e_z$. Suppose $e_z\subset\pa\Om_1$ for $z\in\pa\Om_1$.
 It is easy to check that $\norm{\psi_z}_{L^\infty(\Om_z)}\lesssim h_{e_z}^{-(n-1)}$,  $\norm{\na\psi_z}_{L^\infty(\Om_z)}\lesssim h_{e_z}^{-n}$, where $\Om_z:=\supp(\phi_z)$, and
\begin{equation}\label{ec5'}
Z_hv=v \quad\forall\, v\in W_h^0.
\end{equation}
This operator enjoys the following stability and interpolation estimates
(see \cite{scott1990finite}):
%
\begin{lemma}\label{lapp} For any $K\in {\cal M}_h$, we have
 \begin{align}
  &  \norm{Z_hv}_{L^\infty(K)}\lesssim \norm{v}_{L^\infty(\tilde{K})},\, \norm{\na Z_hv}_{L^\infty(K)}\lesssim \norm{\na v}_{L^\infty(\tilde{K})},\label{E:g1}\\
  & \norm{v-Z_hv}_{L^2(K)}+h_K\norm{v-Z_hv}_{H^1(K)}\lesssim h_K^2\abs{v}_{H^2(\tilde{K})},\label{E:g2}\\
  & \norm{v-Z_hv}_{L^\infty(K)}\ls h_K^2\abs{v}_{W^{2,\infty}(\tilde{K})}, \label{E:g3}
 \end{align}
where  $\tilde{K}$ is the union of all elements
in $\M_h$ having nonempty intersection with $K$.
\end{lemma}

Moreover, we need the following error estimate between $u_1$ and its Scott-Zhang interpolant which uses only the regularity of the homogenization solution $u_0$.
  \begin{lemma}\label{lapp1}  For any $K\in {\cal M}_h$, we have
 \begin{equation}\label{elapp4}
 \begin{split}
    \abs{u_1-Z_hu_1}_{H^1(K)}&\ls  \Big(h_K+\ep+\frac{h_K^2}{\ep}\Big)\abs{u_0}_{H^2(\tilde{K})}+\frac{h_K^{n/2+1}}{\ep}\abs{u_0}_{W^{1,\infty}(K)},
\end{split}
 \end{equation}
 where $\tilde{K}$ is defined in Lemma~\ref{lapp}.
 \end{lemma}
\begin{proof} Denote by $v_j:=\frac{\pa u_0}{\pa x_j}$.
It is easy to see that
\begin{align*}
\abs{u_1-Z_hu_1}_{H^1(K)}\ls& \abs{u_0-Z_h u_0}_{H^1(K)}+\ep\abs{\chi^j v_j -Z_h \left(\chi^j v_j \right)}_{H^1(K)} \\
\ls&\abs{u_0-Z_h u_0}_{H^1(K)}+\ep\abs{(\chi^j-Z_h\chi^j) v_j }_{H^1(K)}\\
&+\ep\abs{Z_h\chi^j\left( v_j -\pd{v_j} _K\right)}_{H^1(K)}
+\ep\abs{Z_h\chi^j\pd{v_j} _K-Z_h \left(\chi^j v_j \right)}_{H^1(K)}\\
:=&\mathrm{I}+\mathrm{II}+\mathrm{III}+\mathrm{IV}.
\end{align*}
where $ \langle\cdot\rangle_K = \frac{1}{|K|} \int_K (\cdot) \rd x.$
From \eqref{E:g2}, we have
\begin{align}
    \mathrm{I}&\ls h_K\abs{u_0}_{H^2(\tilde{K})}.  \label{sz1}
\end{align}
From the facts that $\abs{\na\chi^j}\ls \ep^{-1}$, $\abs{\na^2\chi^j}\ls \ep^{-2}$ and \eqref{E:g1}--\eqref{E:g3}, we have
\begin{align}
\mathrm{II}&\ls \frac{h_K^{n/2+1}}{\ep}\abs{u_0}_{W^{1,\infty}(K)}+\frac{h_K^2}{\ep}\abs{u_0}_{H^2(K)},\\
 \mathrm{III}&\ls (h_K+\ep)\abs{u_0}_{H^2(K)}, \label{sz3}
\end{align}
where we have used the Poincar\'{e} inequality to derive the second inequality. It remains to estimate IV.  According to the definition of Scott-Zhang interpolation, we have
\begin{align*}
\mathrm{IV}&=\ep\abs{\sum_{\text{ node }z\in K} \phi_z(x)\int_{e_z}\psi_z\chi^j \left(\pd{v_j} _K - v_j \right) }_{H^1(K)} \\
&\ls \ep h_K^{n/2} \abs{\sum_{\text{ node }z\in K} \na \phi_z(x)\int_{e_z}\psi_z\chi^j \left(\pd{v_j} _K - v_j \right) }\\
&\ls \ep h_K^{-n/2}\sum_{\text{ node }z\in K}\sum_{j=1}^n\int_{e_z}\abs{\pd{v_j} _K - v_j } .
\end{align*}
For each node $z\in K$, there exist a number $M_z\ls 1$ and a sequence of elements $K_{z,m}\subset \tilde K, m=1,\cdots, M_z,$ such that $K_{z,1}=K$, $K_{z,i}$ and $K_{z,i+1}$ have a common edge/face $e_{z,i}, i=1,\cdots,M_z-1$, and $e_{z,M_z}:=e_z$ is an edge/face of $K_{z,M_z}$. Clearly, we have
\begin{align*}
\int_{e_z}\abs{\pd{v_j} _K - v_j }\le& \int_{e_{z,M_z}}\abs{\pd{v_j} _{K_{z,M_z}} - v_j }+h_{e_{z,M_z}}\abs{\pd{ v_j }_{K_{z,1}}-\pd{ v_j }_{K_{z,M_z}}}\\
\le&\int_{e_{z,M_z}}\abs{\pd{v_j} _{K_{z,M_z}} - v_j }\\
&+h_{e_{z,M_z}}\bigg|\pd{ v_j }_{K_{z,1}}-\frac{1}{h_{z,1}}\int_{e_{z,1}}v_j+\frac{1}{h_{z,1}}\int_{e_{z,1}}\big(v_j-\pd{ v_j }_{K_{z,2}}\big)\\
& \qquad\qquad +\cdots+\frac{1}{h_{z,M_z-1}}\int_{e_{z,M_z-1}} \big(v_j-\pd{ v_j }_{K_{z,M_z}}\big)\bigg|\\
\ls&\sum_{m=1}^{M_z}\int_{\pa K_{z,m}}\abs{\pd{v_j} _{K_{z,m}} - v_j }\\
\ls&h_K^{(n-1)/2}\sum_{m=1}^{M_z}\norml{\pd{v_j} _{K_{z,m}} - v_j}{\pa K_{z,m}}.\\
\end{align*}
Thus, by the trace inequality and Poincar\'{e} inequality, we obtain
\begin{align*}
\mathrm{IV}&\ls\ep h_K^{-1/2}\sum_{\text{ node }z\in K}\sum_{j=1}^n\sum_{m=1}^{M_z}\bigg(h_K^{-1/2}\norml{\pd{v_j} _{K_{z,m}} - v_j }{K_{z,m}}\\
&\hskip 130pt +h_K^{1/2}\norml{\na\big(\pd{v_j} _{K_{z,m}} - v_j \big)}{K_{z,m}} \bigg)\\
&\ls\ep \sum_{j=1}^n\norml{\na v_j}{\tilde K}\ls \ep\abs{u_0}_{H^2(\tilde K)},
\end{align*}
which, combining with \eqref{sz1}--\eqref{sz3}, yields the result immediately.
\end{proof}

From Lemma~\ref{lapp1} and Theorem~\ref{homoerror}, we  have the following result which gives $H^1$ approximation estimates of the space $W_h$. The proof is omitted.
 \begin{lemma}\label{lapp2} Let $\hat u_h:=Z_hu_1$. Then
\begin{align}\label{lapph1}
   \abs{u_1-\hat u_h}_{H^1(\Om_1)}&\ls \ep\abs{u_0}_{H^2(\Om_1)}
   +\frac{h}{\ep}\abs{\Om_1}^{1/2}\abs{u_0}_{W^{1,\infty}(\Om_1)},\\
      \abs{u_1-\hat u_h}_{H^1(\Om_{\Ga_h})}&\ls \ep\abs{u_0}_{H^2(\Om_1)}
   +\frac{h^{3/2}}{\ep}\abs{u_0}_{W^{1,\infty}(\Om_{\Ga_h})},\label{lapph2}
   \end{align}
 where $\Om_{\Ga_h}$ is defined in \eqref{eogHh}.
\end{lemma}

\section{Error estimates for the FE-MsFEM}\label{sec-5}
In this section we derive the  $H^1$-error estimate for the FE-MsFEM in the case where $\beta=1$. For other cases such that $\beta=0, -1$, the analysis is similar and is omitted here.
Since the convergence analysis is only done for the periodic coefficient case, we will fix $\rho=\ep$ in the later analysis.

The following Lemma gives an inverse estimate for the function in space $\text{OMS}(K)$.
\begin{lemma}\label{inverse1}
Assume that $v_H\in \text{OMS}(K)$. Then, we have
\begin{equation}\label{inverse1est}
    \abs{v_H}_{H^2(K)}\ls ({\ep}^{-1}+H^{-1})\abs{v_H}_{H^1(K)}.
\end{equation}
\end{lemma}
\begin{proof}
Assume that $v_H=c_i\bar{\psi}_i^K$. By the definition of $\bar{\psi}_i^K$, we can extend the $v_H$ to the macro element $S\supset K$ as following:
 ${v_H}=c_ic_{ij}^K\psi_j^S$, and  ${\Pi_K v_H}=c_ic_{ij}^K\varphi_j^S$, where $c_{ij}^K, \psi_j^S$, and $\varphi_j^S$ are defined in Section~\ref{sec-2}( see \eqref{overbase}--\eqref{coefcij2} ).  It is easy to verify that ${v_H}$ satisfies
\begin{equation*}
\left\{\begin{aligned}
-\nabla\cdot(\mathbf{a}^\ep(x)\nabla
{v_H}(x)) & = 0,&&
  x\in S, \\
  {v_H}&={\Pi_K v_H},&&   x\in\partial S.
\end{aligned}\right.
\end{equation*}
Hence, from Lemma~\ref{interiorestimate} and $\text{dist} (K,\partial S)\gtrsim H_K$, it follows that
 \begin{equation*}
|v_H|_{H^2(K)}\ls \left({\ep}^{-1}+{H}^{-1}\right)\norm{\na v_H}_{L^2(S)},
\end{equation*}
which yields
 \begin{equation*}
|{v_H}|_{H^2(K)}\ls \left({\ep}^{-1}+{H}^{-1}\right)\norm{\na \Pi_K v_H}_{L^2(S)}\ls (\ep^{-1}+H^{-1})\norm{\na {\Pi_K v_H}}_{L^2(K)}.
\end{equation*}
Therefore, from Lemma~\ref{L:8.8}, it follows the result \eqref{inverse1est} immediately.
\end{proof}

The following lemma gives the continuity and coercivity of the
bilinear form $A_\beta(\cdot,\cdot)$ for the FE-MsFEM.
 \begin{lemma}\label{lbil} We have
 \begin{align}\label{econt}
    \abs{A_\beta(v,w)}\le 2\normc{v}\normc{w} \quad\forall\, v,w\in V_{h,H}.
 \end{align}
For any $0<\ga_1\lesssim 1$, there exists a constant $\al_0$ independent of $h$, $H$, $\ep$, and the penalty parameters such that, if $\ga_0\ge\al_0\big/\ga_1$, then
\begin{equation}\label{ecoer}
     A_\beta(v_{h,H},v_{h,H})\ge \frac12\normc{v_{h,H}}^2 \quad\forall\, v_{h,H}\in V_{h,H}.
\end{equation}
\end{lemma}
\begin{proof}
  \eqref{econt} is a  direct consequence of the definitions \eqref{eah}--\eqref{eJ1}, \eqref{enorma}, \eqref{enormb}, and the Cauchy-Schwarz inequality.

  It remains to prove \eqref{ecoer}. We have,
  \begin{align*}
     &A_\beta(v_{h,H},v_{h,H})=\sum_{K\in {\cal M}_{h,H}}\norml{(\mathbf{a}^\ep)^{1/2}\na v_{h,H}}{K}^2\hskip-2pt-2\sum_{e\in\Ga_h}\int_e \av{\mathbf{a}^\ep\na v_{h,H}\cdot\bn}\jm{v_{h,H}}\\
    &\hskip 70pt+\sum_{e\in\Ga_h}\left( \frac{\ga_0 }{\ep}\norml{\jm{v_{h,H}}}{e}^2+\ga_1\ep\norml{\jm{\mathbf{a}^\ep\na v_{h,H}\cdot\bn}}{e}^2 \right)\\
&\quad=\normc{v_{h,H}}^2-\sum_{e\in\Ga_h} \frac{\ep}{\ga_0}\norml{\av{\mathbf{a}^\ep\na v_{h,H}\cdot\bn}}{e}^2 -2\sum_{e\in\Ga_h}\int_e \av{\mathbf{a}^\ep\na v_{h,H}\cdot
\bn}\jm{v_{h,H}}.
  \end{align*}
  It is obvious that,
  \begin{align*}
    2\sum_{e\in\Ga_h}&\int_e \av{\mathbf{a}^\ep\na v_{h,H}\cdot\bn}\jm{v_{h,H}}\le 2\sum_{e\in\Ga_h}\norml{\av{\mathbf{a}^\ep\na v_{h,H}\cdot\bn}}{e}\norml{\jm{v_{h,H}}}{e}\\
    &\le \sum_{e\in\Ga_h} \frac{\ga_0 }{2 \ep}\norml{\jm{v_{h,H}}}{e}^2+\sum_{e\in\Ga_h}\frac{2 \ep}{\ga _0 }\norml{\av{\mathbf{a}^\ep\na v_{h,H}\cdot\bn}}{e}^2.\nn
  \end{align*}
  Therefore,
  \begin{align*}\label{elbil1}
     A_\beta(v_{h,H},v_{h,H})\ge&\normc{v_{h,H}}^2-\hskip -2pt\sum_{e\in\Ga_h} \frac{\ga_0 }{ 2\ep}\norml{\jm{v_{h,H}}}{e}^2-\hskip -2pt\sum_{e\in\Ga_h}\frac{3 \ep}{\ga _0 }\norml{\av{\mathbf{a}^\ep\na v_{h,H}\cdot\bn}}{e}^2.\nn
  \end{align*}
It is clear that, for any $e\in\Ga_h$,
 \begin{equation*}
    \av{\mathbf{a}^\ep\na v_{h,H}\cdot\bn}\Big|_{e}=(\mathbf{a}^\ep\na v_H)\cdot\bn+\frac{1}{2}\jm{\mathbf{a}^\ep\na v_{h,H}\cdot\bn}\Big|_{e}.
 \end{equation*}
 We have
   \begin{align}\label{elbil1}
     A_\beta(v_{h,H},v_{h,H})\ge&\normc{v_{h,H}}^2-\sum_{e\in\Ga_h} \frac{\ga_0 }{ 2\ep}\norml{\jm{v_{h,H}}}{e}^2\\
     &-\frac{6 \ep}{\ga _0 }\sum_{e\in\Ga_h}\Big(\norml{(\mathbf{a}^\ep\na v_H)\cdot\bn}{e}^2+\frac14\norml{\jm{\mathbf{a}^\ep\na v_{h,H}\cdot\bn}}{e}^2\Big).\nn
  \end{align}
By the trace inequality, the inverse estimate \eqref{inverse1est},  and $\ep<H$, we have
  \begin{align}\label{elbil5}
    &\sum_{e\in E}\norml{(\mathbf{a}^\ep\na v_H)\cdot\bn}{e}^2=\norml{(\mathbf{a}^\ep\na v_H)\cdot\bn}{E}^2 \\
    &\qquad\qquad\quad\leq C\left(\frac{1}{H_E}\norml{\na v_H}{K_E}^2+\norml{\na v_H}{K_E}\norml{\na^2 v_H}{K_E}\right)\nn\\
    &\qquad\qquad\quad\leq \frac{C_1}{\ep}\norm{(\mathbf{a}^\ep)^{1/2}\na v_H}_{L^2(K_E)}^2,\nn
  \end{align}
  where $K_E\in\M_H$ is the element containing $E$.
Therefore, from \eqref{elbil1} and \eqref{elbil5}, it follows that
   \begin{align*}
     A_\beta(v_{h,H},v_{h,H})\ge&\normc{v_{h,H}}^2-\sum_{e\in\Ga_h} \frac{\ga_0 }{ 2\ep}\norml{\jm{v_{h,H}}}{e}^2\\
     &- \frac{6C_1}{\ga_0}\norm{v_H}_{1,H}^2-\frac{3}{2\ga_0\ga_1 }\sum_{e\in\Ga_h}\ga_1\ep \norml{\jm{\mathbf{a}^\ep\na v_{h,H}\cdot\bn}}{e}^2\\
     \ge&\normc{v_{h,H}}^2-\max\left(\frac12,\frac{6C_1\ga_1}{\ga_0\ga_1}, \frac{3}{2\ga_0\ga_1 }\right)\normc{v_{h,H}}^2.
  \end{align*}
Noting that $\ga_1\lesssim 1$, there exists a constant $\al_0>0$ independent of $h$, $H$, $\ep$ such that if $\ga_0\ga_1\ge\al_0$ then $\max\left(\frac{6C_1\ga_1}{\ga_0\ga_1}, \frac{3}{2\ga_0\ga_1 }\right)\le\frac12$.
This completes the proof of the Lemma.
\end{proof}

The following lemma is an analogue of the Strang's lemma for nonconforming finite element methods.
 \begin{lemma}\label{lcea}
 There exists a constant $\al_{0}$ independent of $\ep$, $h$, $H$, and the penalty parameters such that for $0<\ga_1\lesssim 1$, $\ga_0\ge\al_0/\ga_1$, the following error estimate holds:
 \begin{equation}\label{noconforming}
 \begin{split}
 &\normc{u_\ep-u_{h,H}}\\
 &\quad\lesssim \ \inf_{v_{h,H}\in
 V_{h,H}}\normc{u_\ep-v_{h,H}}+\sup_{w_{h,H}\in
 V_{h,H}}\frac{\abs{\int_\Om f
 w_{h,H}\rd x-A_\beta(u_\ep,w_{h,H})}}{\normc{w_{h,H}}}.
 \end{split}
 \end{equation}
\end{lemma}

 \begin{proof} For any $v_{h,H}\in V_{h,H}$, from Lemma~\ref{lbil},  \eqref{ecoer}, \eqref{eipfem}, and \eqref{econt}, we have,
  for $0<\ga_1\lesssim 1$ and $\ga_0\ge\al_0/\ga_1$,
  \begin{align*}
 &\normc{u_{h,H}-v_{h,H}}^2\lesssim A_\beta(u_{h,H}-v_{h,H},u_{h,H}-v_{h,H})\\
 &\quad= (f,u_{h,H}-v_{h,H})-A_\beta(v_{h,H},u_{h,H}-v_{h,H})\\
 &\quad=(f,u_{h,H}-v_{h,H})-A_\beta(u_\ep,u_{h,H}-v_{h,H})+A_\beta(u_\ep-v_{h,H},u_{h,H}-v_{h,H})\\
 &\quad\lesssim (f,u_{h,H}-v_{h,H})-A_\beta(u_\ep,u_{h,H}-v_{h,H})+\normc{u_\ep-v_{h,H}}\normc{u_{h,H}-v_{h,H}}.
 \end{align*}
 Hence
  \begin{align*}
 &\normc{u_\ep-u_{h,H}}\leq \normc{u_\ep-v_{h,H}}+\normc{u_{h,H}-v_{h,H}}\\
& \qquad\lesssim \ \normc{u_\ep-v_{h,H}}+\frac{\abs{(f,u_{h,H}-v_{h,H})-A_\beta(u_\ep,u_{h,H}-v_{h,H})}}{\normc{u_{h,H}-v_{h,H}}}
\end{align*}
which yields the error estimate \eqref{noconforming}. This completes the proof.
 \end{proof}

Now, we are ready to present the main result of the paper which gives the error estimate in the  norm $\normc{\cdot}$ for the FE-MsFEM.
\begin{theorem}\label{energeerror} Assume that
the penalty parameter  $0<\ga_1\lesssim 1$ and $\ga_0\ge \al_0/\ga_1$. Then the following error estimate holds:
\begin{align*}
\normc{u_\eps-u_{h,H}} & \lesssim \left( \sqrt{\ep}+\frac{\ep}{H}+\frac{h}{\ep}\abs{\Om_1}^{1/2}\right)\abs{u_0}_{W^{1,\infty}(\Om)}+ H\norm{f}_{L^2(\Om)}\\
 &\quad+\frac{H^{2}}{\sqrt{\ep}}\frac{\abs{u_0}_{H^2(\Om_{\Ga})}}{\sqrt{\abs{\Om_{\Ga}}}}+ \ep^2\norm{\na f}_{L^2(\Om_\Ga)},
\end{align*}
where $\Om_\Ga$   is defined in \eqref{eog}.
\end{theorem}
\begin{remark} (a) The error bound consists of three parts: the first part of order $O\big( \sqrt{\ep}$ $+\frac{\ep}{H}+ H\big)$ from the oversampling MsFE approximation in $\Om_2$, the second part of order $ O\left(\dfrac{h}{\ep}\abs{\Om_1}^{1/2}\right)$ from the FE approximation in $\Om_1$, and the third part $\dfrac{H^{2}}{\sqrt{\ep}}\dfrac{\abs{u_0}_{H^2(\Om_{\Ga})}}{\sqrt{\abs{\Om_{\Ga}}}}+\ep^2\norm{\na f}_{L^2(\Om_\Ga)}$ from the penalizations on $\Ga$.

(b) Suppose that the interface $\Ga$ is chosen such that $\dist(\Ga,\pa\Om)=O(H)$. If the average value
\begin{equation}\label{eau0}
\dfrac{\abs{u_0}_{H^2(\Om_{\Ga})}}{\sqrt{\abs{\Om_{\Ga}}}}\ls 1,
\end{equation}
then we have
$$\normc{u_\eps-u_{h,H}}  \lesssim \sqrt{\ep}+\frac{\ep}{H}+ H + \frac{hH^{1/2}}{\ep}+\frac{H^{2}}{\sqrt{\ep}}.
$$
In this case, we may choose $H\eqsim\sqrt{\eps}$ and $h\eqsim\ep^{5/4}$ to ensure that  $\normc{u_\eps-u_{h,H}} \lesssim \sqrt{\ep}$.
The condition \eqref{eau0} may be checked by using the standard singularity decomposition results for elliptic problems on polygonal domains \cite{Grisvard, Dauge98}. For example, we may show for the two dimensional case (n=2) that,
if the inner angles of the polygon $\Om$ are less than $\frac23\pi$, then \eqref{eau0} holds.
\end{remark}
\begin{proof} According to Lemma~\ref{lcea}, the proof is divided into two parts. The first part is devoted to estimating the interpolation error and the second part to estimating the non-conforming error.

{\it Part 1. Interpolation error estimate}. We set $v_{h,H}$ as $v_{h,H}|_{\Om_1}=\hat u_h, v_{h,H}|_{\Om_2}=\psi_H$, where $\hat u_h:=Z_hu_1$ and $\psi_H$ are defined in Lemma~\ref{lapp2} and Lemma~\ref{multiapp} respectively. We are going to estimate $\normc{u_\eps-v_{h,H}}$, i.e.,  to estimate each term in its definition (cf. \eqref{enormb}). First,  from Lemmas~\ref{multiapp1},~\ref{lapp2}, we have
\begin{align}\label{errornorm}
&\left(\norm{u_\ep-v_{h,H}}_{1,H}^2+ \norml{(\mathbf{a}^\ep)^{1/2}\na (u_\ep-v_{h,H})}{\Om_1}^2\right)^{1/2}\\
&\ls H\abs{u_0}_{H^2(\Om)}+\sqrt{\ep}\abs{u_0}_{W^{1,\infty}(\Om)}  +\frac{\eps}{H}\abs{u_0}_{W^{1,\infty}(\Om_2)}+\frac{h}{\ep}\abs{\Om_1}^{1/2}\abs{u_0}_{W^{1,\infty}(\Om_1)}.\nn
\end{align}
Further, since $\jm{u_\ep}=0$ and $\jm{{u}_1}=0$, it is easy to see that
\[
\begin{split}
\sum_{e\in\Ga_h} \frac{\ga_0 }{\ep}\int_e\jm{u_\ep-v_{h,H}}^2  & = \sum_{e\in\Ga_h} \frac{\ga_0 }{\ep}\int_e\jm{{u}_1-v_{h,H}}^2 \\
&\leq \sum_{E\in\Ga_H} \frac{\ga_0 }{\ep}\int_E(u_1-\psi_H)^2 +\sum_{e\in\Ga_h} \frac{\ga_0 }{\ep}\int_e (u_1-\hat u_h)^2 .
\end{split}
\]
By the trace inequality, we have
\[
\int_E(u_1-\psi_H)^2 \lesssim H^{-1}\norm{u_1-\psi_H}_{L^2(K)}^2 +\norm{u_1-\psi_H}_{L^2(K)}\norm{\na(u_1-\psi_H)}_{L^2(K)}.
\]
Hence, taking a summation over $\Ga_H$, yields
\[
\begin{split}
 \sum_{E\in\Ga_H}& \frac{\ga_0 }{\ep}\int_E(u_1-\psi_H)^2 \lesssim   \frac{1}{H\ep} \sum_{K\in K_{\Ga_H}}\norm{u_1-\psi_H}_{L^2(K)}^2\\
   & + \frac{1}{\ep} \bigg(\sum_{K\in K_{\Ga_H}} \norm{u_1-\psi_H}_{L^2(K)}^2\bigg)^{1/2}\bigg(\sum_{K\in K_{\Ga_H}} \norm{\na(u_1-\psi_H)}_{L^2(K)}^2\bigg)^{1/2}.
\end{split}
\]
Thus, from  from \eqref{maximumtheta}, \eqref{multiapph2}, and \eqref{multiapph3}, it follows that
\[
\begin{split}
 \sum_{E\in\Ga_H} \frac{\ga_0 }{\ep}\int_E(u_\ep-\psi_H)^2 \lesssim & \ep\abs{u_0}_{W^{1,\infty}(\Om_{\Ga_H})}^2+\frac{H^3}{\ep}\abs{u_0}_{H^2(\Om_{\Ga_H})}^2.
\end{split}
\]
Further, by the trace inequality, from Lemma~\ref{lapp}, it follows
\[
\begin{split}
 &\sum_{e\in\Ga_h} \frac{\ga_0 }{\ep}\int_e(u_1-\hat u_h)^2 \\
 &\qquad\lesssim \frac{1}{\ep}\sum_{e\in\Ga_h}\left( h^{-1}\norm{u_1-\hat u_h}_{L^2(K_e)}^2+\norm{u_1-\hat u_h}_{L^2(K_e)}\norm{\na(u_1-\hat u_h)}_{L^2(K_e)}\right)\\
 &\qquad \ls \frac{h^3}{\ep}\abs{u_1}_{H^2(\Om_{\Ga_h})}^2,
 \end{split}
\]
where $\Om_{\Ga_h}$ is defined in \eqref{eogHh}, and, in order to ensure the last inequality, for any vertex of elements in $K_{\Ga_h}$, we have chosen the corresponding edge/face in the definition of Scott-Zhang interpolation to be an edge/face of some element in $K_{\Ga_h}$.
Therefore, from the above two estimates and Lemma~\ref{uelocalh2}, we have
\begin{align}\label{jumperror}
\sum_{e\in\Ga_h} \frac{\ga_0 }{\ep}\int_e\jm{u_\ep-v_{h,H}}^2  \ls&\frac{H^3}{\ep}\abs{u_0}_{H^{2}(\Om_{\Ga})}^2+h^2\abs{u_0}_{H^2(\Om)}^2\\
&+{\ep}\abs{u_0}_{W^{1,\infty}(\Om)}^2+\ep^4\norml{\na f}{\Om_\Ga}^2.\nn
\end{align}

Next, we estimate the term
\[
\sum_{e\in\Ga_h} \frac{\ep}{\ga_0}\norml{\av{\mathbf{a}^\ep\na (u_\ep-v_{h,H})\cdot\bn}}{e}^2.
\]
It is easy to see that
\[
\begin{split}
\sum_{e\in\Ga_h} \frac{\ep}{\ga_0}\norml{\av{\mathbf{a}^\ep\na (u_\ep-v_{h,H})\cdot\bn}}{e}^2\lesssim &\sum_{E\in\Ga_H} \frac{\ep}{\ga_0}\norml{\mathbf{a}^\ep\na (u_\ep-\psi_H)\cdot\bn}{E}^2\\
&+\sum_{e\in\Ga_h} \frac{\ep}{\ga_0}\norml{\mathbf{a}^\ep\na (u_\ep-\hat u_h)\cdot\bn}{e}^2.
\end{split}
\]
By the trace inequality, we have
\[
\begin{split}
&\norml{\mathbf{a}^\ep\na (u_\ep-\psi_H)\cdot\bn}{E}^2\ls \norml{\na (u_\ep-u_1)\cdot\bn}{E}^2+\norml{\na (u_1-\psi_H)\cdot\bn}{E}^2\\
&\quad\leq H^{-1}\norml{\na (u_\ep-u_1)}{K}^2+\norml{\na (u_\ep-u_1)}{K}\norml{\na^2 (u_\ep-u_1)}{K}\\
&\qquad +H^{-1}\norml{\na (u_1-\psi_H)}{K}^2+\norml{\na (u_1-\psi_H)}{K}\norml{\na^2 (u_1-\psi_H)}{K}.
\end{split}
\]
Hence, a summation over $\Ga_H$ follows that
\[
\begin{split}
&\sum_{E\in\Ga_H} \frac{\ep}{\ga_0}\norml{\mathbf{a}^\ep\na (u_\ep-\psi_H)\cdot\bn}{E}^2\\
&\qquad\lesssim \ep H^{-1}\norml{\na (u_\ep-u_1)}{\Om_{\Ga_H}}^2+\ep H^{-1}\sum_{K\in K_{\Ga_H}} \norml{\na (u_1-\psi_H)}{K}^2\\
 &\qquad\quad+{\ep}\norm{\na (u_\ep-u_1)}_{L^2(\Om_{\Ga_H})}\norm{\na^2(u_\ep-u_1)}_{L^2(\Om_{\Ga_H})}\\
   & \qquad\quad+ {\ep} \bigg(\sum_{K\in K_{\Ga_H}} \norm{\na (u_1-\psi_H)}_{L^2(K)}^2 \bigg)^{1/2} \bigg(\sum_{K\in K_{\Ga_H}} \norm{\na^2(u_1-\psi_H)}_{L^2(K)}^2 \bigg)^{1/2}.
\end{split}
\]
Therefore, it follows that  from Theorems~\ref{homoerror}, \ref{h2homoerror} and Lemma~\ref{multiapp},
\begin{align}\label{intererr1}
\sum_{E\in\Ga_H} \frac{\ep}{\ga_0}\norml{\mathbf{a}^\ep\na (u_\ep-\psi_{H})\cdot\bn}{e}^2\lesssim & (H+\ep)^2\abs{u_0}_{H^2(\Om)}^2\\
&+\ep\abs{u_0}_{W^{1,\infty}(\Om)}^2+\ep^4\norm{\na f}_{L^2(\Om_\Ga)}^2.\nn
\end{align}
Similarly, by the trace inequality, we have
\[
\begin{split}
&\sum_{e\in\Ga_h} \frac{\ep}{\ga_0}\norml{\mathbf{a}^\ep\na (u_\ep-\hat u_h)\cdot\bn}{e}^2\\
&\quad\ls \ep\sum_{E\in\Ga_H}\norml{\mathbf{a}^\ep\na (u_\ep-u_1)\cdot\bn}{E}^2+\ep\sum_{e\in\Ga_h}\norml{\mathbf{a}^\ep\na (u_1-\hat u_h)\cdot\bn}{e}^2\\
&\quad\ls \ep H^{-1}\norml{\na (u_\ep-u_1)}{\Om_{\Ga_H}}^2+\ep\norml{\na (u_\ep-u_1)}{\Om_{\Ga_H}}\norml{\na^2 (u_\ep-u_1)}{\Om_{\Ga_H}}\\
&\qquad+ \ep h^{-1}\norml{\na (u_1-\hat u_h)}{\Om_{\Ga_h}}^2+\ep\norml{\na (u_1-\hat u_h)}{\Om_{\Ga_h}}\norml{\na^2 u_1}{\Om_{\Ga_h}}.
\end{split}
\]
Thus, from Theorems~\ref{homoerror},\ref{h2homoerror} and Lemmas~\ref{uelocalh2}, \ref{lapp2}, it follows that
\begin{align}\label{intererr2}
&\sum_{e\in\Ga_h} \frac{\ep}{\ga_0}\norml{\mathbf{a}^\ep\na (u_\ep-\hat u_h)\cdot\bn}{e}^2 \\
 &\qquad\ls \ep^2\abs{u_0}_{H^2(\Om)}^2+\ep\abs{u_0}_{W^{1,\infty}(\Om)}^2+\ep^4\norm{\na f}_{L^2(\Om_\Ga)}^2.\nn
\end{align}

It is obvious that a same argument as above can be used to get the same error bound for the term
\[
\sum_{e\in\Ga_h}{\ga_1} {\ep}\norml{\jm{\mathbf{a}^\ep\na (u_\ep-v_{h,H})\cdot\bn}}{e}^2. \]
Thus, it follows from \eqref{errornorm}--\eqref{intererr2}  that
\begin{align}\label{comformerror}
\inf_{v_{h,H}\in
 V_{h,H}}&\normc{u_\ep-v_{h,H}}\lesssim H\abs{u_0}_{H^2(\Om)}+\left( \sqrt{\ep}+\frac{\ep}{H}\right)\abs{u_0}_{W^{1,\infty}(\Om)}\\
 &+\frac{h}{\ep}\abs{\Om_1}^{1/2}\abs{u_0}_{W^{1,\infty}(\Om_1)}+\frac{H^{3/2}}{\sqrt{\ep}}\abs{u_0}_{H^2(\Om_{\Ga})}+ \ep^2\norm{\na f}_{L^2(\Om_\Ga)}.\nn
\end{align}

{\it Part 2. The non-conforming error estimate}.
Define
$$
    \mathcal{E}_{{H}}^{\mathrm{I}}:=  \hbox{set of all interior edges/faces  of ${\cal M}_H$ }.
$$
For any $w_{h,H}\in V_{h,H}$, noticing that $\pa\Om_2/\Ga$ is empty, it is easy to see
\[
\begin{split}
(f,w_{h,H}) &=(f,w_h)_{\Om_1}+(f,w_H)_{\Om_2}=(-\na\cdot(\mathbf{a}^\ep\na u_\ep),w_h) _{\Om_1}+(-\na\cdot(\mathbf{a}^\ep\na u_\ep),w_H)_{\Om_2}\\
&=(\mathbf{a}^\ep\na u_\ep,\na w_h)_{\Om_1}-\int_\Ga \mathbf{a}^\ep\na u_\ep\cdot \mathbf{n} w_h \\
&\qquad\quad +\sum_{K\in {\cal M}_H}\left((\mathbf{a}^\ep\na u_\ep,\na w_H)_K-\int_{\pa K}\mathbf{a}^\ep\na u_\ep\cdot \mathbf{n}_K w_H \right)\\
&=(\mathbf{a}^\ep\na u_\ep,\na w_h)_{\Om_1}-\sum_{e\in\Ga_{h}}\int_e \mathbf{a}^\ep\na u_\ep\cdot \mathbf{n} (w_h-w_H) \\
&\qquad +\sum_{K\in {\cal M}_H}(\mathbf{a}^\ep\na u_\ep,\na w_H)_K-\sum_{E\in  \mathcal{E}_{{H}}^{\mathrm{I}}}\int_E \mathbf{a}^\ep\na u_\ep\cdot \mathbf{n}_E [w_H].
\end{split}
\]
Here the unit normal vector $\mathbf{n}_E$ is oriented from $ K$ to $K'$ and the jump $[v]$ of $v$ on an interior side $E=\partial K\cap\partial K'\in \mathcal{E}_H^{\mathrm{I}}$ is defined as $[v]:=v|_K-v|_{K'}$. Furthermore, by the definition of $A_\be$, it follows that
$$
A_\be(u_\ep,w_{h,H})=(\mathbf{a}^\ep\na u_\ep,\na w_h)_{\Om_1}-\sum_{e\in\Ga_{h}}\int_e \mathbf{a}^\ep\na u_\ep\cdot \mathbf{n} (w_h-w_H) +\sum_{K\in {\cal M}_H}(\mathbf{a}^\ep\na u_\ep,\na w_H)_K.
$$
Thus, we have
\[
\begin{split}
(f,w_{h,H})&-A_\be(u_\ep,w_{h,H})= -\sum_{E\in  \mathcal{E}_{{H}}^{\mathrm{I}}}\int_E \mathbf{a}^\ep\na u_\ep\cdot \mathbf{n}_E \jm{w_H}\\
&= -\sum_{E\in  \mathcal{E}_{{H}}^{\mathrm{I}}}\int_E \mathbf{a}^\ep\na u_\ep\cdot \mathbf{n}_E \jm{w_H-\Pi_H w_H}\\
&=-\sum_{K\in {\cal M}_H}\int_{\pa K}\mathbf{a}^\ep\na u_\ep\cdot \mathbf{n}_K (w_H-\Pi_H w_H) -\sum_{E\in \Ga_{H}}\int_E \mathbf{a}^\ep\na u_\ep\cdot \mathbf{n} (w_H-\Pi_H w_H)\\
&:= \mathrm{R}_1 +\mathrm{R}_2.
\end{split}
\]
Since
$$
\mathrm{R}_1=\int_{\Om_2}f(w_H-\Pi_H w_H) -\sum_{K\in\M_H}\int_K\mathbf{a}^\ep\na
 u_{\eps}\na(w_H-\Pi_Hw_H) ,
$$
we can estimate $\mathrm{R}_1$ by following \cite[Chapter 6]{EH2009}, or  \cite[Chapter 9]{CW2010}, or
the proof presented in \cite[Theorem 3.1]{EHW}, and obtain
\begin{equation}\label{errorR1}
\abs{\mbox{\rm R}_1}\leq
C\Big(\eps\abs{u_0}_{H^2(\Om_2)}+\Big(\sqrt{\ep}+\frac{\eps}{H}\Big)\abs{u_0}_{W^{1,\infty}(\Om_2)}+\ep\norm{f}_{L^2(\Om_2)}\Big)\norm{w_H}_{1,H}.
\end{equation}

Next, we consider the second term $\mathrm{R}_2$.
For any $w_H\in X_H$, it is easy to check that (see \cite{CW2010,EHW})
   \begin{equation*}
       w_H=\Pi_Kw_H+\eps\chi^j\frac{\pa(\Pi_Kw_H)}{\pa x_j}+\eps\theta_\eps^S \quad\mathrm{in}\quad S,
   \end{equation*}
where $\theta_\eps^S\in H^1(S)$ is the boundary corrector given by
   \begin{equation*}
    -\na\cdot(\mathbf{a}^\eps\na\theta_\eps^S)=0\quad \text{ in } S,\qquad \theta_\eps^S\big|_{\pa S}=-\chi^j\frac{\pa(\Pi_Kw_h)}{\pa x_j}.
   \end{equation*}
By the Maximum Principle, we have
\begin{equation*}
\|\theta_\ep^S\|_{L^\infty(S)}\leq C\|\nabla
\Pi_Kw_H\|_{L^\infty(S)}.
\end{equation*}
Thus, from Lemma~\ref{L:8.8}, we obtain
\begin{equation*}
\begin{split}
\|w_H-\Pi_Kw_H\|_{L^\infty(\pa K)}&\ls\ep\|\nabla \Pi_Kw_H\|_{L^\infty(S)}\ls\ep\|\nabla \Pi_Kw_H\|_{L^\infty(K)}\\
&\ls\ep H^{-n/2}\|\nabla \Pi_Kw_H\|_{L^2(K)}\ls\ep H^{-n/2}\|\na w_H\|_{L^2(K)},
\end{split}
\end{equation*}
which yields
\begin{align*}
\abs{\mathrm{R}_2}&\ls \sum_{E\in \Ga_{H}}\norm{\na u_\ep}_{L^{2}(E)}\norm{w_H-\Pi_Hw_H}_{L^2(E)}\\
&\ls \norm{\na u_\ep}_{L^{2}(\Ga)} \bigg(\sum_{E\in \Ga_{H}}H^{n-1}\norm{w_H-\Pi_Hw_H}_{L^\infty(E)}^2 \bigg)^{1/2}\\
&\ls \norm{\na u_\ep}_{L^{2}(\Ga)}\bigg(\sum_{K\in K_{\Ga_H}}\ep^2 H^{-1}\|\na w_H\|_{L^2(K)}^2\bigg)^{1/2}\\
&\ls \ep H^{-1/2}\norm{\na u_\ep}_{L^{2}(\Ga)}\norm{w_H}_{1,H}.
\end{align*}
On the other hand, by use of the trace inequality, it follows from Theorems~\ref{homoerror}, \ref{h2homoerror},  and Lemma~\ref{h3u0} that
\begin{equation*}
\begin{split}
&\norm{\na u_\ep}_{L^2(\Ga)}^2\le \norm{\na (u_\ep-u_1)}_{L^2(\Ga)}^2+\norm{\na u_0}_{L^2(\Ga)}^2+\norm{\na \left(\ep\chi^i\frac{\pa u_0}{\pa x_j}\right)}_{L^2(\Ga)}^2\\
&\quad \ls\sum_{K\in K_{\Ga_H}}\Big(H^{-1} \norm{\na (u_\ep-u_1)}_{L^2(K)}^2+\norm{\na (u_\ep-u_1)}_{L^2(K)}\abs{u_\ep-u_1}_{H^2(K)}\Big)\\
&\qquad +\abs{u_0}_{W^{1,\infty}(\Om)}^2+\ep^2\sum_{K\in K_{\Ga_H}}\Big(H^{-1}\abs{u_0}_{H^2(K)}^2+\abs{u_0}_{H^2(K)}\abs{u_0}_{H^3(K)}\Big)\\
&\quad\ls\ep\abs{u_0}_{H^2(\Om)}^2+\abs{u_0}_{W^{1,\infty}(\Om)}^2+\ep^3\norm{\na f}_{L^2(\Om_\Ga)}^2.
\end{split}
\end{equation*}
Therefore
\begin{equation}\label{errorR2}
\abs{\mathrm{R}_2}\ls\Big(\ep\abs{u_0}_{H^2(\Om)}+\sqrt{\ep}\abs{u_0}_{W^{1,\infty}(\Om)}+\ep^2\norm{\na f}_{L^2(\Om_\Ga)}\Big)\norm{w_H}_{1,H}.
\end{equation}
It follows from \eqref{errorR1} and \eqref{errorR2} that the non-conforming error in Lemma~\ref{lcea}
\begin{align*}
&\sup_{w_{h,H}\in
 V_{h,H}}\frac{\abs{\int_\Om f
 w_{h,H} -A_\beta(u,w_{h,H})}}{\normc{w_{h,H}}}\\
&\qquad\ls \eps\abs{u_0}_{H^2(\Om)}+\Big(\sqrt{\ep}+\frac{\eps}{H}\Big)\abs{u_0}_{W^{1,\infty}(\Om)}+\ep\norm{f}_{L^2(\Om)}+\ep^2\norm{\na f}_{L^2(\Om_\Ga)},
\end{align*}
which, combining  with \eqref{comformerror}, \eqref{noconforming}, and \eqref{regularityh2}, completes the proof.
\end{proof}

\section{Numerical tests}\label{sec-6}
In this section,
we first demonstrate the performance of the proposed FE-MsFEM by solving the model problem \eqref{eproblem} with periodic and randomly generated coefficients respectively, and then show the ability of the FE-MsFEM to solve two multiscale elliptic problems with high-contrast channels. In all computations we do not assume that the diffusion coefficient values are available outside of the research domain.
In order to illustrate the performance of our method, we also implement two other kinds of methods. The first is the standard MsFEM. The second one is a mixed basis MsFEM which use the oversampling multiscale basis inside the domain but away from the boundary, while use the standard MsFEM basis near the boundary. By this way, the mixed basis MsFEM doesn't need to use the outside information.

For the methods FE-MsFEM and mixed basis MsFEM, the triangulation may be done by the following three steps.
\begin{itemize}
 \item First, we triangulate the domain $\Omega$ with a coarse mesh whose mesh size $H$ is much bigger than  $\ep$.
 \item Secondly, we choose the union of coarse-grid elements adjacent to the boundary $\pa\Om$ (and the channels if exist) as $\Om_1$ and denote $\Om\setminus\overline{\Om_1}$ by $\Om_2$. For example, in our tests, we choose two layers of coarse-grid elements (and the coarse-grid elements containing the channels if exist) to form the domain $\Om_1$. Hence the distance of $\Ga$ away from $\pa\Om$ is $2H$.
 \item Finally, in $\Om_2$, we use the oversampling MsFEM basis on coarse-grid elements. While, in $\Om_1$ we use the traditional linear FEM basis on a fine mesh for the FE-MsFEM,  or use the standard MsFEM basis on coarse-grid elements for the mixed basis MsFEM. In our tests, we fix the mesh size of the fine mesh $h=1/1024$ which is small enough to resolve the smallest scale of oscillations.
 \end{itemize}
Please see Fig.~\ref{fig:1} for a sample triangulation.

Since there are no exact solutions to the problems considered here, we will solve them on a very fine mesh with mesh size $h_f=1/2048$ by use of the traditional linear finite element method, and consider their numerical solutions as the ``exact" solutions which are denoted as $u_e$. Denoting by $u_h$ the numerical solutions computed by the methods considered in this section, we measure the relative errors in the $L^2$, $L^\infty$ and energy norms as following
\[
\frac{\norm{u_h-u_e}_{L^2}}{\norm{u_e}_{L^2}},\,\frac{\norm{u_h-u_e}_{L^\infty}}{\norm{u_e}_{L^\infty}},
\,\frac{\norme{u_h-u_e}}{\norme{u_e}}.
\]
In all tests, for simplicity,  the penalty parameters in our FE-MsFEM are chosen as $\ga_0=20$ and $\ga_1=0.1$. The coefficient $\mathbf{a}^\ep$ is chosen as the form $\mathbf{a}^\ep=a^\ep I$ where $a^\ep$ is a scalar function and $I$ is the 2 by 2 identity matrix.

\subsection{Application to elliptic problems with highly oscillating coefficients}

We first consider the model problem (\ref{eproblem}) in the  squared domain
$\Omega=(0,1)\times(0,1)$. Assume that $f=1$ and the coefficient
$\mathbf{a}^\ep(x_1,x_2)$ has
 the following periodic form
\begin{equation}\label{coef1}
a^\ep(x_1,x_2)=\frac{2+1.8\sin(2\pi x_1/\ep)}{2+1.8\cos(2\pi x_2/\ep)}
+\frac{2+1.8\sin(2\pi x_2/\ep)}{2+1.8\sin(2\pi x_1/\ep)}.
\end{equation}
where we fix $\ep=1/100$. In our FE-MsFEM, we consider two choices of the parameter $\rho$. The first choice is $\rho=\eps$ as stated in our theoretical analysis, while the other one $\rho=h$, the size of the fine mesh. The second choice is useful when the scales are non-separable. We  first choose $H=1/32$ and report the relative errors in the $L^2$, $L^\infty$ and energy norms in Table~\ref{Table:1}.
\begin{table}[htp]
\caption{Relative errors in the $L^2$, $L^\infty$ and energy norms for the model problem with periodic coefficient given by \eqref{coef1}. $\ep=1/100$, $H=1/32, h=1/1024$, $\ga_0=20$, $\ga_1=0.1$.  }\label{Table:1}
\begin{center}
\begin{tabular}{|c|c|c|c|} \hline
 {Relative Error}  &  $L^2$ & $L^\infty$ &Energy norm \\\hline
   MsFEM           &  0.7263e-01 & 0.7157e-01 & 0.2560e-00\\\hline
   Mixed basis MsFEM         &  0.3422e-01 & 0.3637e-01 & 0.1714e-00\\\hline
   FE-MsFEM $\rho=\eps$       &  0.1238e-01 & 0.1334e-01 & 0.5159e-01\\\hline
   FE-MsFEM $\rho=h$       &  0.1252e-01 & 0.1344e-01 & 0.4840e-01
\\\hline
\end{tabular}
\end{center}
\end{table}
We can see that the FE-MsFEMs give the most accurate results among the methods considered here. Especially, when we take $\rho=h$, the FE-MsFEM still works well.

The following numerical experiment is to show the coarse mesh size $H$ plays a role as that describing in the Theorem~\ref{energeerror}. We fix $h=1/1024$ and $\ep=1/100$. Three kinds of coarse mesh size are chosen. The first one, $H=1/64$, is denoted as $64\times16$; the second one, $H=1/32$, is denoted as $32\times32$; the last one, $H=1/16$, is denoted as $16\times64$. The results are shown in Table~\ref{Table:2}.
\begin{table}[htp]
\caption{Relative errors of the FE-MsFEM with $H=1/64, 1/32,$ and $1/16$, respectively, for the model problem with periodic coefficient given by \eqref{coef1}. $\rho=\ep=1/100$, $h=1/1024$, $\ga_0=20$, $\ga_1=0.1$.}\label{Table:2}
\begin{center}
\begin{tabular}{|c|c|c|c|} \hline
 {Relative Error}  &  $L^2$ & $L^\infty$ &Energy norm \\\hline
   $64\times16$    &  0.1100e-01 & 0.1502e-01 & 0.9342e-01\\\hline
   $32\times32$    &  0.1186e-01 & 0.1290e-01 & 0.5159e-01\\\hline
   $16\times64$    &  0.1240e-01 & 0.1795e-01 & 0.6593e-01
\\\hline
\end{tabular}
\end{center}
\end{table}
From the table, it is easy to see that as $H$ goes larger, the relative error in energy norm goes lower first and goes higher later, which is coincided with the theoretical results in Theorem~\ref{energeerror}.

Next we simulate the model problem with a random coefficient which is generated by using
the random log-normal permeability field $\mathbf{a}^\ep(x)$ by using the
moving ellipse average technique \cite{dur} with the variance of the
logarithm of the permeability $\sigma^2=1.5$, and the correlation
lengths $l_1=l_2=0.01$ (isotropic heterogeneity) in $x_1$ and $x_2$
directions, respectively. The ratio of maximum to minimum of one
realization of the resulting permeability field in our numerical
experiments is 1.6137e+05.  One realization of the resulting
permeability field in our numerical experiments is depicted in
Fig.~\ref{perms11}.
\begin{figure}[htp]
\centerline{\includegraphics[scale=0.6]{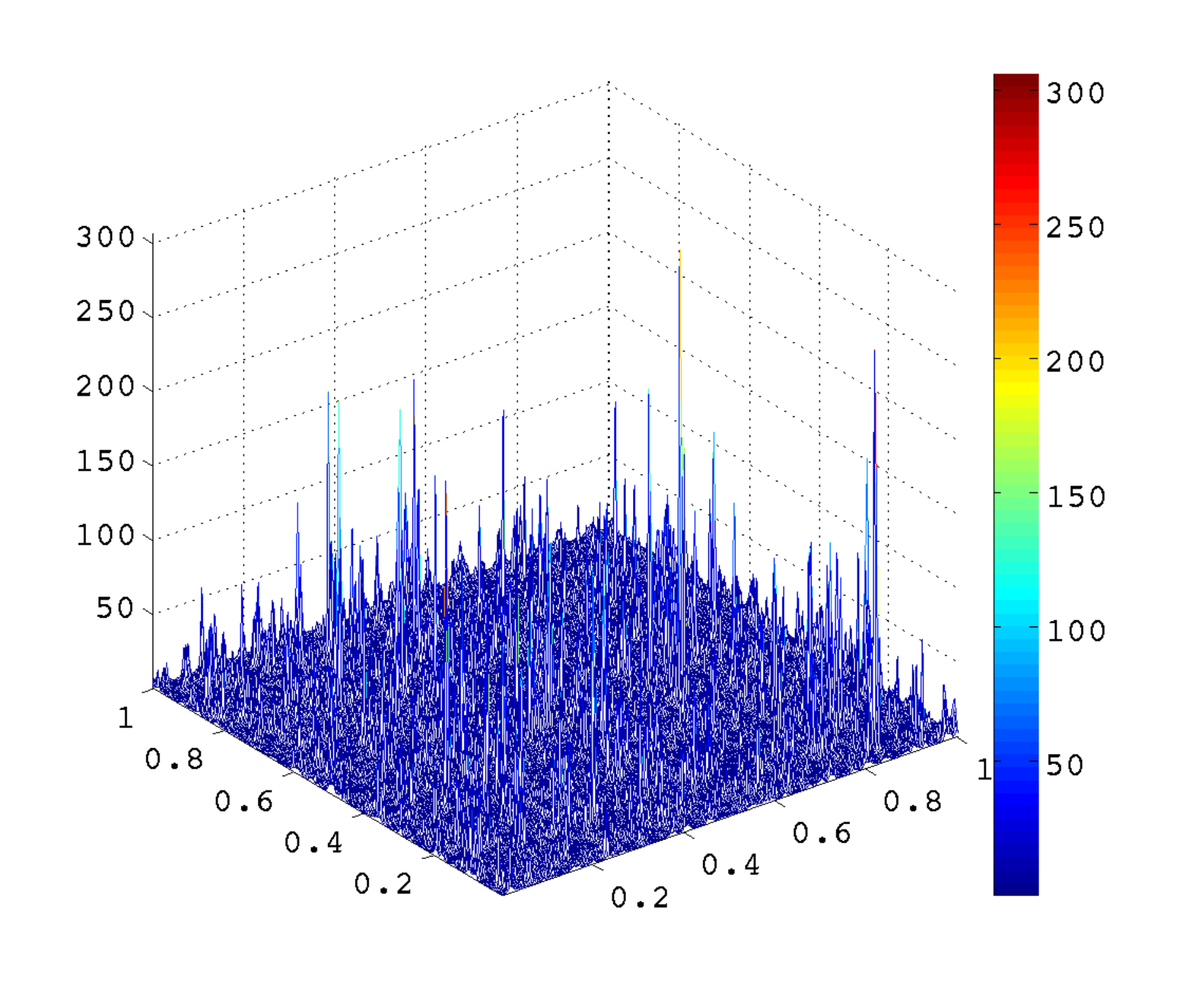}}
\caption{\label{perms11}The random log-normal permeability field
$a^\ep(x)$. The ratio of maximum to minimum is 1.6137e+05.}
\end{figure}
We also compare three kinds of methods including the standard MsFEM, the Mixed basis MsFEM and the FE-MsFEM. In this test, we set $H=1/32$ and $\rho=h$ since there is no explicit $\eps$ in this example. The relative errors for the three methods are listed in Table~\ref{Table:3}. From the table, we can also see that the FE-MsFEM gives the most accurate results among the methods considered here.
\begin{table}[htp]
\caption{Relative errors in the $L^2$, $L^\infty$ and energy norms for the model problem with random coefficient as shown in Figure~\ref{perms11}. $H=1/32, \rho=h=1/1024$, $\ga_0=20$, $\ga_1=0.1$. }\label{Table:3}
\begin{center}
\begin{tabular}{|c|c|c|c|} \hline
 {Relative Error}  &  $L^2$ & $L^\infty$ &Energy norm \\\hline
   MsFEM           &  0.3690e-00 & 0.3731e-00 & 0.6014e-00\\\hline
   Mixed basis MsFEM         &  0.1119e-00 & 0.1857e-00 & 0.4770e-00\\\hline
   FE-MsFEM        &  0.2635e-01 & 0.8351e-01 & 0.2975e-00
\\\hline
\end{tabular}
\end{center}
\end{table}

\subsection{Application to multiscale problems with high-contrast channels}
In this subsection, we use the introduced FE-MsFEM to solve two elliptic multiscale problems which have high-contrast channels inside the domain.

 In the first example, the coefficient $\mathbf{a}^\ep$ is characterized
by a fine and long-ranged high-permeability channel, which is set by the following way.  The example utilizes the periodic coefficient $a^\ep$ in \eqref{coef1} as the background, while changing the values on a narrow and long channel that defined from $(0.08, 0.49)$ to $(0.92, 0.51)$ with new value $a^\ep=10^5$ (See Fig.~\ref{fig:2}).
\begin{figure}[htp]
\centerline{\includegraphics[width=0.5\textwidth]{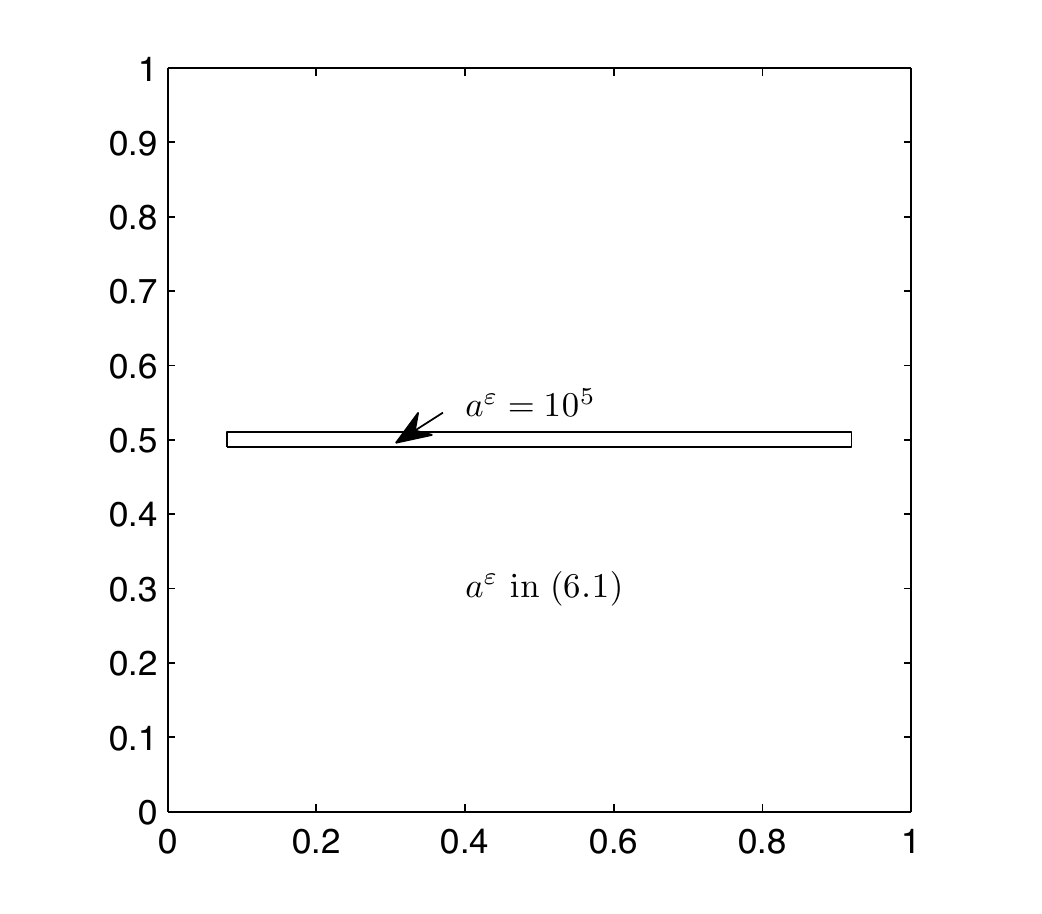}}
\caption{Permeability field}
\label{fig:2}
\end{figure}
For this problem, the ``exact " solution is difficult to be obtained due to the singularities near the corners of the high contrast channel. Our direct numerical simulation shows that the gradient values are large near the corners of the channel.
We set $H=1/32$ and $\rho=\ep=1/100$. The results are presented in Table~\ref{Table:4} where the relative errors in $L^2,L^\infty$ norms as well as energy norm are shown. We observe that the FE-MsFEM  performs better than other methods.
\begin{table}[htp]
\caption{Relative errors for the model problem with the permeability depicted in Fig.\ref{fig:2}. $\rho=\ep=1/100$, $H=1/32$, $h=1/1024$, $\ga_0=20$, $\ga_1=0.1$. }\label{Table:4}
\begin{center}
\begin{tabular}{|c|c|c|c|} \hline
 {Relative Error}  &  $L^2$ & $L^\infty$ &Energy norm \\\hline
   MsFEM    &  0.1640e-00 & 0.2187e-00 & 0.3773e-00\\\hline
   Mixed basis MsFEM    &  0.5415e-01 & 0.2552e-00 & 0.2977e-00\\\hline
   FE-MsFEM     &  0.1127e-01 & 0.2090e-01 & 0.6843e-01
\\\hline
\end{tabular}
\end{center}
\end{table}

In the second example, we use the coefficient depicted in Fig~\ref{fig:3} that corresponds to a coefficient with background one and high permeability channels and inclusions with permeability values equal to $10^5$ and $8\times10^4$ respectively.
\begin{figure}[htp]
\centerline{\includegraphics[width=0.5\textwidth]{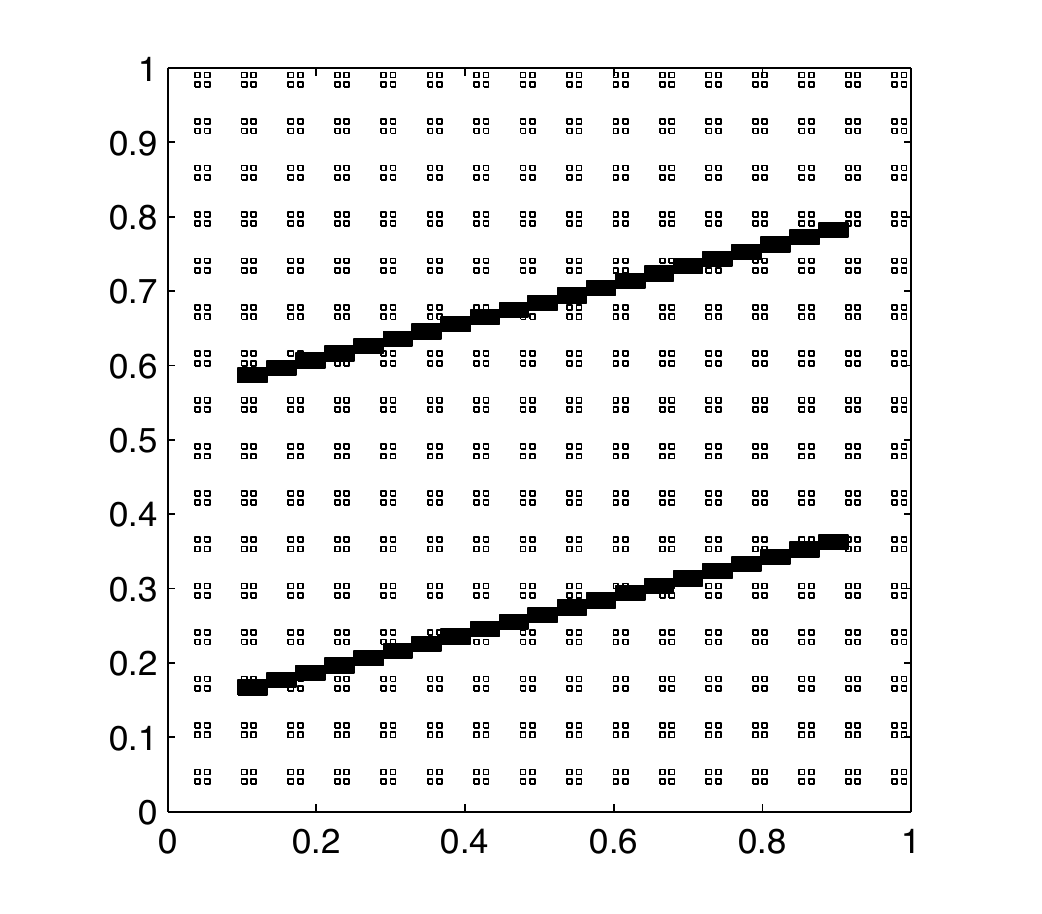}}
\caption{Permeability field: $a^\ep=10^5$ in two channels consisting of dark small rectangles; $a^\ep=8\times10^4$ in small square inclusions; $a^\ep=1$ otherwise. }
\label{fig:3}
\end{figure}
The results are listed in Table~\ref{Table:5}. We observe that our FE-MsFEM  gives much better results than the other two methods.
\begin{table}[htp]
\caption{Relative errors for the model problem with the permeability depicted in Fig.\ref{fig:3}. $H=1/32$, $\rho=h=1/1024$, $\ga_0=20$, $\ga_1=0.1$.}\label{Table:5}
\begin{center}
\begin{tabular}{|c|c|c|c|} \hline
 {Relative Error}  &  $L^2$ & $L^\infty$ &Energy norm \\\hline
   MsFEM    &  0.3546e-00 & 0.4007e-00 & 0.5943e-00\\\hline
   Mixed basis MsFEM    &  0.2243e-00 & 0.2596e-00 & 0.4997e-00\\\hline
   FE-MsFEM     & 0.4274e-02 & 0.1284e-01 & 0.7566e-01\\\hline
\end{tabular}
\end{center}
\end{table}

\section{Conclusions}
In this paper, we have developed a new numerical scheme for the elliptic multiscale problems which joints the oversampling MsFEM and the standard FEM together by using the penalty techniques. The idea is first to separate the research domain into two parts $\Om_1$ and $\Om_2=\Om\setminus\overline{\Om_1}$ that $\Om_1$ contains the boundary $\pa\Om$ where the oversampling MsFEM can not apply, and singular points (or regions) where the oversampling MsFEM is inefficient. Then we apply the standard FEM on a fine mesh of $\Om_1$ and the oversampling MsFEM on a coarse mesh of $\Om_2$. The two methods are jointed on the interface $\Ga=\pa\Om_1\cap\pa\Om_2$ of the fine and coarse meshes by penalizing the jumps of the function values as well as the fluxes of discrete solutions.

A rigorous and careful analysis has been given for the elliptic equation with periodic diffusion coefficient to show that, under some mild assumptions, if $\Ga$ is so chosen that $\dist(\Ga,\pa\Om)\gtrsim H$, then the $H^1$-error of our new method is of order
\[O\Big(\sqrt{\ep}+\frac{\ep}{H}+ H +\frac{h}{\ep}|\Om_1|+\frac{H^{2}}{\sqrt{\ep}}\Big),\]
which exactly consists of the oversampling MsFE approximation error in $\Om_2$, the FE approximation error in $\Om_1$, and the  error contributed by the penalizations on $\Ga$. Note that, for simplicity, we have only analyzed the linear version of FEM for the discretization on $\Om_1$.

Numerical experiments are carried out for the elliptic equations with periodic oscillating or random coefficients, as well as,  the multiscale problems with high contrast channels, to verify the theoretical findings and compare the performance of our FE-MsFEM with the standard MsFEM and Mixed basis MsFEM. It is shown that, the FE-MsFEM performs better than the other two methods in all cases and much better in some experiments.

There are several ways to improve further the performance of our FE-MsFEM. First, the linear FEM on $\Om_1$ can be apparently extended to higher order FEMs to reduce the error term related to $\Om_1$. Secondly, since $\Om_1$ may contains singularities, another interesting project is to consider a combination of adaptive FEM on local refined meshes on $\Om_1$ and oversampling MsFEM on $\Om_2$. Thirdly, based on existence numerical results for oversampling MsFEMs \cite{HW}, we  conjecture that the theoretical assumption of $\dist(\Ga,\pa\Om)\gtrsim H$ may be weaken to  $\dist(\Ga,\pa\Om)\ge C\ep$ (at least, in practice) for some constant $C$. These will be left as future studies.


\begin{thebibliography}{10}

\bibitem{Aarnes2004}
{\sc J.~E. Aarnes}, {\em On the use of a mixed multiscale finite element method
  for greater flexibility and increased speed or improved accuracy in reservoir
  simulation}, SIAM MMS, 2 (2004), pp.~421--439.

\bibitem{Arbo1}
{\sc T.~Arbogast}, {\em Numerical subgrid upscaling of two-phase flow in porous
  media}, in Numerical Treatment of Multiphase Flows in Porous Media, Z.~Chen,
  R.~E. Ewing, and Z.~C. Shi, eds., vol.~552 of Lect. Notes Phys.,
  Springer-Verlag, New York, 2000, pp.~35--49.

\bibitem{Arbo2}
\leavevmode\vrule height 2pt depth -1.6pt width 23pt, {\em Implementation of a
  locally conservative numerical subgrid upscaling scheme for two-phase darcy
  flow}, Comput. Geosci., 6 (2002), pp.~453--481.

\bibitem{APWY2007}
{\sc T.~Arbogast, G.~Pencheva, M.~F. Wheeler, and I.~Yotov}, {\em A multiscale
  mortar mixed finite element method}, SIAM MMS, 6 (2007), pp.~319--346.

\bibitem{arnold82}
{\sc D.~Arnold}, {\em An interior penalty finite element method with
  discontinuous elements}, SIAM J. Numer. Anal., 19 (1982), pp.~742--760.

\bibitem{abcm01}
{\sc D.~Arnold, F.~Brezzi, B.~Cockburn, and D.~Marini}, {\em Unified analysis
  of discontinuous {G}alerkin methods for elliptic problems.}, SIAM J. Numer.
  Anal., 39 (2001), pp.~1749--1779.

\bibitem{al87}
{\sc M.~Avellaneda and F.~Lin}, {\em Compactness methods in the theory of
  homogenization}, Comm. Pure Appl. Math., 40 (1987), pp.~803--847.

\bibitem{BCO1994}
{\sc I.~Babuska, G.~Caloz, and J.~Osborn}, {\em Special finite element methods
  for a class of second order elliptic problems with rough coefficients}, SIAM
  J. Numer. Anal., 31 (1994), pp.~945--981.

\bibitem{BO1983}
{\sc I.~Babuska and J.~Osborn}, {\em Generalized finite element methods: their
  performance and their relation to mixed methods}, SIAM J. Numer.Anal., 20
  (1983), pp.~510--536.

\bibitem{bz73}
{\sc I.~Babu\v{s}ka and M.~Zl\'amal}, {\em Nonconforming elements in the finite
  element method with penalty}, SIAM Journal on Numerical Analysis, 10 (1973),
  pp.~pp. 863--875.

\bibitem{baker77}
{\sc G.~Baker}, {\em Finite element methods for elliptic equations using
  nonconforming elements}, Math. Comp., 31 (1977), pp.~44--59.

\bibitem{BLP}
{\sc A.~Bensoussan, J.~L. Lions, and G.~Papanicolaou}, {\em Asymptotic analysis
  for periodic structure}, vol.~5 of Studies in Mathematics and Its
  Application, North-Holland Publ., 1978.

\bibitem{BFHR}
{\sc F.~Brezzi, L.~P. Franca, T.~J.~R. Hughes, and A.~Russo}, {\em $b = \int
  g$}, Comput. Methods Appl. Mech. Engrg., 145 (1997), pp.~329--339.

\bibitem{CH2002}
{\sc Z.~Chen and T.~Y. Hou}, {\em A mixed multisclae finite method for elliptic
  problemswith oscillating coefficients}, Math. Comp., 72 (2002), pp.~541--576.

\bibitem{CW2010}
{\sc Z.~Chen and H.~Wu}, {\em Selected topics in finite element method},
  Science Press, Beijing, 2010.

\bibitem{CY2002}
{\sc Z.~Chen and X.~Y. Yue}, {\em Numerical homogenization of well
  singularities in the flow transport through heterogeneous porous media}, SIAM
  MMS, 1 (2003), pp.~260--303.

\bibitem{Dauge98}
{\sc M.~Dauge}, {\em Elliptic Boundary Value Problems in Corner Domains --
  Smoothness and Asymptotics of Solutions.}, vol.~1341 of Lecture Notes in
  Mathematics, Springer-Verlag, Berlin, 1988.

\bibitem{DE}
{\sc M.~Dorobantu and B.~Engquist}, {\em Wavelet-based numerical
  homogenization}, SIAM J. Numer. Anal., 35 (1998), pp.~540--559.

\bibitem{dd76}
{\sc J.~Douglas~Jr and T.~Dupont}, {\em Interior Penalty Procedures for
  Elliptic and Parabolic {G}alerkin methods}, Lecture Notes in Phys. 58,
  Springer-Verlag, Berlin, 1976.

\bibitem{dur}
{\sc L.~Durlofsky}, {\em Numerical calculation of equivalent grid block
  permeability tensors for heterogeneous porous media}, Water Resources
  Research, 27 (1991), pp.~699--708.

\bibitem{EE1}
{\sc W.~E and B.~Engquist}, {\em The heterogeneous multiscale methods}, Commun.
  Math. Sci., 1 (2003), pp.~87--132.

\bibitem{EE3}
\leavevmode\vrule height 2pt depth -1.6pt width 23pt, {\em Multiscale modeling
  and computation}, Notice Amer. Math. Soc., 50 (2003), pp.~1062--1070.

\bibitem{EMZ}
{\sc W.~E, P.~Ming, and P.~Zhang}, {\em Analysis of the heterogeneous
  multiscale method for elliptic homogenization problems}, J. Am. Math. Soc.,
  18 (2005), pp.~121--156.

\bibitem{EGW2011}
{\sc Y.~Efendiev, J.~Galvis, and X.~H. Wu}, {\em Multiscale finite element
  methods for high-contrast problems using local spectral basis functions}, J.
  Comput. Phys., 230 (2011), pp.~937--955.

\bibitem{EGHE2006}
{\sc Y.~Efendiev, V.~Ginting, T.~Y. Hou, and R.~Ewing}, {\em Accurate
  multiscale finite element methods for two-phase flow simulations}, J. Comput.
  Phys., 220 (2006), pp.~155--174.

\bibitem{EH2009}
{\sc Y.~Efendiev and T.~Y. Hou}, {\em Multiscale finite element methods theory
  and applications}, Springer, Lexington, KY, 2009.

\bibitem{EHG2004}
{\sc Y.~Efendiev, T.~Y. Hou, and V.~Ginting}, {\em Multiscale finite element
  methods for nonlinear partial differential equations}, Comm. Math. Sci., 2
  (2004), pp.~553--589.

\bibitem{EHW}
{\sc Y.~Efendiev, T.~Y. Hou, and X.~H. Wu}, {\em Convergence of a nonconforming
  multiscale finite element method}, SIAM J. Numer. Anal., 37 (2000),
  pp.~888--910.

\bibitem{EP2004}
{\sc Y.~Efendiev and A.~Pankov}, {\em Numerical homogenization of nonlinear
  random parabolic operators}, SIAM MMS, 2 (2004), pp.~237--268.

\bibitem{ER}
{\sc B.~Engquist and O.~Runborg}, {\em Wavelet-based numerical homogenization
  with applications}, in Multiscale and Multiresolution Methods: Theory and
  Applications, T.~Barth, T.~Chan, and R.~Heimes, eds., vol.~20 of Lecture
  Notes in Computational Sciences and Engineering, Springer-Verlag, Berlin,
  2002, pp.~97--148.

\bibitem{FHF}
{\sc C.~Farhat, I.~Harari, and L.~P. Franca}, {\em The discontinuous enrichment
  method}, Comput. Meth. Appl. Mech. Eng., 190 (2001), pp.~6455--6479.

\bibitem{Farmer}
{\sc C.~L. Farmer}, {\em Upscaling: A review}, in Proceedings of the Institute
  of Computational Fluid Dynamics Conference on Numerical Methods for Fluid
  Dynamics, Oxford, UK, 2001.

\bibitem{fw09}
{\sc X.~Feng and H.~Wu}, {\em Discontinuous {G}alerkin methods for the
  {H}elmholtz equation with large wave numbers.}, SIAM J. Numer. Anal., 47
  (2009), pp.~2872--2896, also downloadable at {\tt
  http://arXiv.org/abs/0810.1475}.

\bibitem{fw11}
\leavevmode\vrule height 2pt depth -1.6pt width 23pt, {\em $hp$-discontinuous
  {G}alerkin methods for the {H}elmholtz equation with large wave number},
  Math. Comp.,  (2011, posted online).

\bibitem{FB2}
{\sc J.~Fish and V.~Belsky}, {\em Multigrid method for a periodic heterogeneous
  medium, part i: Multiscale modeling and quality in multidimensional case},
  Comput. Meth. Appl. Mech. Eng., 126 (1995), pp.~17--38.

\bibitem{FY}
{\sc J.~Fish and Z.~Yuan}, {\em Multiscale enrichment based on partition of
  unity}, Inter. J. Numer. Meth. Eng., 62 (2005), pp.~1341--1359.

\bibitem{FR}
{\sc L.~P. Franca and A.~Russo}, {\em Deriving upwinding, mass lumping and
  selective reduced integration by residual-free bubbles}, Appl. Math. Lett., 9
  (1996), pp.~83--88.

\bibitem{GE20101}
{\sc J.~Galvis and Y.~Efendiev}, {\em Domain decomposition preconditioners for
  multiscale flows in high-contrast media}, SIAM MMS, 8 (2010), pp.~1461--1483.

\bibitem{GE20102}
\leavevmode\vrule height 2pt depth -1.6pt width 23pt, {\em Domain decomposition
  preconditioners for multiscale flows in high-contrast media: Reduced
  dimension coarse spaces}, SIAM MMS, 8 (2010), pp.~1621--1644.

\bibitem{GT}
{\sc D.~Gilbarg and N.~Trudinger}, {\em Elliptic partial differential equations
  of second order}, Springer-Verlag, Berlin, 2001.

\bibitem{Grisvard}
{\sc P.~Grisvard}, {\em Elliptic problems on nonsmooth domains}, Pitman,
  Boston, 1985.

\bibitem{HW}
{\sc T.~Y. Hou and X.~H. Wu}, {\em A multiscale finite element method for
  elliptic problems in composite materials and porous media}, J. Comput. Phys.,
  134 (1997), pp.~169--189.

\bibitem{HWC}
{\sc T.~Y. Hou, X.~H. Wu, and Z.~Cai}, {\em Convergence of a multiscale finite
  element method for elliptic problems with rapidly oscillation coefficients},
  Math. Comp., 68 (1999), pp.~913--943.

\bibitem{HWZh}
{\sc T.~Y. Hou, X.~H. Wu, and Y.~Zhang}, {\em Removing the cell resonance error
  in the multiscale finite element method via a petrov-galerkin formulation},
  Commun. Math. Sci., 2 (2004), pp.~185--205.

\bibitem{Hu}
{\sc T.~Hughes}, {\em Multiscale phenomena: Green's functions, the dirichlet to
  neumann formulation, subgrid scale models, bubbles and the origin of
  stabilized methods}, Comput. Meth. Appl. Mech. Eng., 127 (1995),
  pp.~387--401.

\bibitem{JLT2003}
{\sc P.~Jenny, S.~Lee, and H.~Tchelepi}, {\em Multi-scale finite-volume method
  for elliptic problems in subsurface flow simulation}, Journal of
  Computational Physics, 187 (2003), pp.~47--67.

\bibitem{JKO}
{\sc V.~V. Jikov, S.~M. Kozlov, and O.~A. Oleinik}, {\em Homogenization of
  differential operators and integral functionals}, Springer-Verlag, Berlin,
  1994.

\bibitem{m09}
{\sc R.~Massjung}, {\em An $hp$-error estimate for an unfitted discontinuous
  {G}alerkin method applied to elliptic interface problems}, RWTH 300, IGPM
  Report, 2009.

\bibitem{moskow1997first}
{\sc S.~Moskow and M.~Vogelius}, {\em First-order corrections to the
  homogenised eigenvalues of a periodic composite medium. a convergence proof},
  PROCEEDINGS-ROYAL SOCIETY OF EDINBURGH A, 127 (1997), pp.~1263--1300.

\bibitem{MDH}
{\sc J.~D. Moulton, J.~E. Dendy, and J.~M. Hyman}, {\em The black box multigrid
  numerical homogenization algorithm}, J. Comput. Phys., 141 (1998), pp.~1--29.

\bibitem{OZh2007}
{\sc H.~Owhadi and L.~Zhang}, {\em Metric-based upscaling}, Communications on
  pure and applied mathematics, 60 (2007), pp.~675--723.

\bibitem{OZh2011}
\leavevmode\vrule height 2pt depth -1.6pt width 23pt, {\em Localized bases for
  finite-dimensional homogenization approximations with nonseparated scales and
  high contrast}, SIAM Multiscale Modeling and Simulation, 9 (2011),
  pp.~1373--1398.

\bibitem{PWY2002}
{\sc M.~Peszy{\'n}ska, M.~Wheeler, and I.~Yotov}, {\em Mortar upscaling for
  multiphase flow in porous media}, Computational Geosciences, 6 (2002),
  pp.~73--100.

\bibitem{San}
{\sc G.~Sangalli}, {\em Capturing small scales in elliptic problems using a
  residual-free bubbles finite element method}, SIAM Multiscale Model. Simul.,
  1 (2003), pp.~485--503.

\bibitem{scott1990finite}
{\sc R.~SCOTT and S.~Zhang}, {\em Finite element interpolation of nonsmooth
  functions satisfying boundary conditions}, Mathematics of Computation, 54
  (1990), pp.~483--493.

\bibitem{w78}
{\sc M.~F. Wheeler}, {\em An elliptic collocation-finite element method with
  interior penalties}, SIAM J. Numer. Anal., 15 (1978), pp.~152--161.

\bibitem{w}
{\sc H.~Wu}, {\em Pre-asymptotic error analysis of {CIP-FEM} and {FEM} for
  {H}elmholtz equation with high wave number. {P}art {I}: Linear version}, to
  appear.

\bibitem{wx10}
{\sc H.~Wu and Y.~Xiao}, {\em An unfitted $hp$-interface penalty finite element
  method for elliptic interface problems}, Submitted.

\bibitem{WEH}
{\sc X.~H. Wu, Y.~Edendiev, and T.~Hou}, {\em Analysis of hmm absolute
  permeability}, Discrete and Continuous Dynamical Systems-series B,  (2002),
  pp.~185--204.

\bibitem{zw}
{\sc L.~Zhu and H.~Wu}, {\em Pre-asymptotic error analysis of {CIP-FEM} and
  {FEM} for {H}elmholtz equation with high wave number. {P}art {II}: $hp$
  version}, to appear.

\end{thebibliography}
\end{document}